\documentclass[intlimits,11pt]{amsart}
\usepackage[T1]{fontenc}
\usepackage[utf8]{inputenc}
\usepackage{latexsym,amsthm,amsmath,amssymb,mathrsfs,mathtools}
plus 6pt minus 12pt
\def\mvint_#1{\mathchoice
          {\mathop{\vrule width 6pt height 3 pt depth -2.5pt
                  \kern -9pt \intop}\limits_{\kern -3pt #1}}%
          {\mathop{\vrule width 5pt height 3 pt depth -2.6pt
                  \kern -6pt \intop}\nolimits_{#1}}%
          {\mathop{\vrule width 5pt height 3 pt depth -2.6pt
                  \kern -6pt \intop}\nolimits_{#1}}%
          {\mathop{\vrule width 5pt height 3 pt depth -2.6pt
                  \kern -6pt \intop}\nolimits_{#1}}}
\newcommand{\overbar}[1]{\mkern 1.7mu\overline{\mkern-1.7mu#1\mkern-1.5mu}\mkern 1.5mu}

\newcommand{\N}{\mathbb N}
\newcommand{\R}{\mathbb R}
\newcommand{\cR}{\mathcal R}
\newcommand{\cH}{\mathcal H}
\newcommand{\bbbr}{\mathbb R}
\newcommand{\bbbb}{\mathbb B}

\newcommand{\bg}{\bar{\gamma}}
\newcommand{\Id}{\mathrm{Id}}
\newcommand{\Sph}{\mathbb S}
\newcommand{\bbbn}{\mathbb N}
\newcommand{\eps}{\varepsilon}
\newcommand{\M}{\mathcal{M}}
\def\dist{\operatorname{dist}}
\def\deg{\operatorname{deg}}
\def\sgn{\operatorname{sgn}}
\def\Tr{\operatorname{Tr}}
\newtheorem{theorem}{Theorem}
\newtheorem*{theorem*}{Theorem}
\newtheorem{lemma}[theorem]{Lemma}
\newtheorem{corollary}[theorem]{Corollary}
\newtheorem{proposition}[theorem]{Proposition}
\theoremstyle{definition}
\newtheorem{remark}[theorem]{Remark}
\newtheorem{definition}[theorem]{Definition}
\newtheorem{question}{Question}

\title{Jacobians of $W^{1,p}$ homeomorphisms, case $p=[n/2]$}
\author[P. Goldstein]{Pawe\l{}  Goldstein}
\address{Pawe\l{} Goldstein, Institute of Mathematics, Faculty of Mathematics, Informatics and Mechanics, University of Warsaw, Banacha 2, 02-097 Warsaw, Poland} \email{P.Goldstein@mimuw.edu.pl}
\thanks{P.G. was partially supported by National Science Center grant no 2012/05/E/ST1/03232.}
\author[P. Haj\l{}asz]{Piotr Haj\l{}asz}
\address{Piotr Haj\l{}asz, Department of Mathematics, University of Pittsburgh,
301 Thackeray Hall, Pittsburgh, PA 15260, USA}
\email{hajlasz@pitt.edu}
\thanks{P.H.\ was supported by NSF grant DMS-1800457}


\begin{document}

%\dsp

\sloppy

\begin{abstract}
We investigate a known problem whether a Sobolev homeomorphism between domains in $\R^n$
can change  sign of the Jacobian.
The only case that remains open is when $f\in W^{1,[n/2]}$, $n\geq 4$. We
prove that if $n\geq 4$, and a sense-preserving homeomorphism $f$ satisfies $f\in W^{1,[n/2]}$, $f^{-1}\in W^{1,n-[n/2]-1}$ and either $f$ is H\"older continuous on
almost all spheres of dimension $[n/2]$, or $f^{-1}$ is H\"older continuous on
almost all spheres of dimensions $n-[n/2]-1$, then the Jacobian of $f$ is non-negative, $J_f\geq 0$, almost everywhere.
This result is a consequence of a more general result proved in the paper. Here $[x]$ stands for the greatest integer less than or equal to~$x$.

\end{abstract}
\maketitle

\centerline{\emph{In memoriam: Bogdan Bojarski (1931-2018)}}

\section{Introduction and results}

A diffeomorphism $f$ between domains in $\mathbb{R}^n$ has either positive or negative Jacobian $J_f=\det Df$. Recall that domains are open and connected. We say that a diffeomorphism is {\em sense preserving} ({\em sense reversing}) if its Jacobian is positive (negative). More generally, we say that a homeomorphism between domains is {\em sense preserving} ({\em sense reversing}) if it has local topological degree $1$ ($-1$) at every point of its domain, see Section~\ref{Sec2.5}. Every homeomorphism is either sense preserving or sense reversing.
It easily follows from the topological properties of the degree that if $f$ is a sense preserving (reversing) homeomorphism of domains in $\R^n$, $f$ is differentiable at $x$ and $J_f(x)\neq 0$, then $J_f(x)>0$ ($J_f(x)<0$). In particular, if a homeomorphism is differentiable almost everywhere, then $J_f\geq 0$ a.e. or $J_f\leq 0$ a.e. However, it may happen that a homeomorphism is differentiable a.e., but its
Jacobian equals zero a.e. For elementary constructions, see \cite{takacs} and references therein.

Let us assume now that a homeomorphism $f$ between domains in $\R^n$ is in the Sobolev space $W^{1,p}_{\rm loc}$, $p\geq 1$. If $p>n-1$, then
$f$ is differentiable a.e., \cite[Corollary~2.25]{HK}, and therefore $J_f\geq 0$ a.e. or $J_f\leq 0$ a.e. Another approach, based on the topological degree, allows one to extend this result to $p\geq n-1$. However, the method completely fails when $1\leq p<n-1$.
Note also that, if $1\leq p<n$, then there are very pathological examples of Sobolev homeomorphisms with the Jacobian equal zero a.e. The first such example was constructed by Hencl \cite{Hencl1}, see also \cite{Cerny,DHS}.

In 2001, Haj\l{}asz asked a question whether a Sobolev $W^{1,p}_{\rm loc}$, $1\leq p<n-1$, homeomorphism between domains in $\R^n$ can change sign of the Jacobian. That is, whether there is a homeomorphism such that $J_f>0$ on a set of positive measure and $J_f<0$ on a set of positive measure.

Hencl and Mal\'y \cite{henclm1} proved the following two results:

\begin{theorem}
\label{T1}
Let $\Omega\subset\R^n$, $n\leq 3$, be a domain and let $f\in W^{1,1}_{\rm loc}(\Omega,\R^n)$ be a sense preserving homeomorphism.
Then $J_f\geq 0$ a.e.
\end{theorem}

\begin{theorem}
\label{T2}
Let $\Omega\subset\R^n$, $n\geq 4$, be a domain and let $f\in W^{1,p}_{\rm loc}(\Omega,\R^n)$, $p>[n/2]$ be a sense preserving homeomorphism. Then $J_f\geq 0$ a.e.
\end{theorem}
Here $[x]$ stands for the greatest integer less than or equal to $x$. Also, by a homeomorphism $f:\Omega\to\R^n$ we mean a homeomorphism onto the image.
Since for $p>[n/2]$ we have the embedding into the Lorentz space
$L_{\rm loc}^p\subset L^{[n/2],1}_{\rm loc}$, Theorem~\ref{T2}
is a special case of a more general result of Hencl and Mal\'y:

\begin{theorem}
\label{T2.5}
Let $\Omega\subset\R^n$, $n\geq 4$, be a domain and let $f:\Omega\to\R^n$
be a sense preserving homeomorphism such that $Df\in L^{[n/2],1}_{\rm loc}$.
Then $J_f\geq 0$ a.e.
\end{theorem}

On the other hand, Hencl and his collaborators, \cite{CHT,henclv}, constructed the following surprising example:

\begin{theorem}
\label{T3}
If $n\geq 4$ and $1\leq p<[n/2]$, then there is a homeomorphism
$f\in W^{1,p}((-1,1)^n,\R^n)$ such that $J_f>0$ on a set of positive measure and $J_f<0$ on a set of positive measure. Moreover, $f$ has the Lusin property.
\end{theorem}
Recall that the Lusin property means that the sets of Lebesgue measure zero are mapped to sets of Lebesgue measure zero.

The result of \cite{henclv} provides such a homeomorphism for $n\geq 4$ with $p=1$, and the general case is obtained in \cite{CHT}.
See also \cite{GH1,GH2} for related examples of approximately differentiable homeomorphisms.

Theorems~\ref{T1},~\ref{T2}, and~\ref{T3} leave only the borderline case open.

\begin{question}
\label{Q1}
Let $\Omega\subset\R^n$, $n\geq 4$, be a domain. Does there exist a homeomorphism
${f\in W^{1,[n/2]}_{\rm loc}(\Omega,\R^n)}$ such that $J_f>0$ on a set of positive measure and $J_f<0$ on a set of positive measure?
\end{question}

The main result of the paper answers this question in the negative under some additional assumptions.

\begin{theorem}
\label{main}
Let
$\Omega\subset\R^{n}$, $n\ge 4$, be a domain and
let $f\in W^{1,[n/2]}_{\rm loc}(\Omega,\R^{n})$ be a sense preserving homeomorphism
such that $f^{-1}\in W^{1,n-[n/2]-1}_{\rm loc}(f(\Omega),\R^{n})$.
Assume also that one of the two conditions is satisfied:
\begin{enumerate}
\item[(a)] $f$ maps almost all spheres of dimension $[n/2]$ to sets of $\mathcal{H}^{[n/2]+1}$-measure zero,
\item[(b)] $f^{-1}$ maps almost all spheres of dimension $n-[n/2]-1$ to sets of $\mathcal{H}^{n-[n/2]}$-measure zero.
\end{enumerate}
Then $J_f\ge 0$ a.e.
\end{theorem}
Here and in what follows, by $\mathcal{H}^k$ we shall denote the $k$-dimensional Hausdorff measure.
\begin{remark}
The space of
$k$-dimensional spheres in $\R^n$ can be parameterized by the product
$G(k+1,n)\times\R^n\times (0,\infty)$, where $G(k+1,n)$ is the Grassmannian of $(k+1)$-dimensional subspaces in $\R^n$. Indeed,  $(V,x,r)\in G(k+1,n)\times\R^n\times (0,\infty)$ defines a sphere centered at $x$, of radius $r$ and parallel to $V$. Since there is a natural measure on $G(k+1,n)\times\R^n\times (0,\infty)$, it makes sense to talk about almost every $k$-dimensional sphere in $\R^n$.
\end{remark}

\begin{remark}
The classes of bi-Sobolev homeomorphisms, i.e., homeomorphisms such that $f$ and $f^{-1}$ belong to Sobolev spaces, have been investigated for example in \cite{CSM,DHS,HK2,HKM,HMPS,Oliva,onninen,Pratelli}.
\end{remark}

The corollaries listed below show particular situations when the condition (a) or (b)
is satisfied.
\begin{corollary}
\label{T13} Let
$\Omega\subset\R^{n}$, $n\ge 4$, be a domain and
let $f\in W^{1,[n/2]}_{\rm loc}(\Omega,\R^{n})$ be a sense preserving homeomorphism
such that $f^{-1}\in W^{1,n-[n/2]-1}_{\rm loc}(f(\Omega),\R^{n})$. Assume also that $f$ or
$f^{-1}$ is H\"older continuous. Then $J_f\ge 0$ a.e.
\end{corollary}
\begin{remark}
In fact, it suffices to assume $f$ is H\"older continuous on almost all $[n/2]$-dimensional spheres or $f^{-1}$ is H\"older continuous on almost all $n-[n/2]-1$ dimensional spheres; the proof remains the same.
\end{remark}
\begin{corollary}
\label{T14} Let
$\Omega\subset\R^{2m}$, $m\geq 2$, be a domain in the even dimensional space
and let $f\in W^{1,m}_{\rm loc}(\Omega,\R^{2m})$ be a sense preserving homeomorphism
such that $f^{-1}\in W^{1,m-1+\eps}_{\rm loc}(f(\Omega),\R^{m})$.
Then $J_f\ge 0$ a.e.
\end{corollary}
\begin{corollary}
\label{T15} Let
$\Omega\subset\R^{2m}$, $m\geq 2$, be a domain in the even dimensional space
and let $f\in W^{1,m}_{\rm loc}(\Omega,\R^{2m})$ be a sense preserving homeomorphism
such that $Df^{-1}\in L^{m-1,1}_{\rm loc}$.
Then $J_f\ge 0$ a.e.
\end{corollary}
\begin{remark}
In Corollaries~\ref{T14} and ~\ref{T15}, we restrict the setting to even dimensions, because a corresponding result in odd dimensions would be a consequence of Theorems~\ref{T2} and~\ref{T2.5} respectively.
\end{remark}
The corollaries easily follow from Theorem~\ref{main}.
\begin{proof}[Proof of Corollary~\ref{T13}]
If $f\in W^{1,k}(\Sph^k,\R^m)$, then there is a set $E\subset \Sph^k$ of measure zero such that the complement of this set is the union of the sets such that on each of these sets $f$ is Lipschitz continuous (see a discussion around \eqref{eq11} below) and hence the Hausdorff dimension of $f(\Sph^k\setminus E)$ is at most $k$. According to a theorem of Mal\'y and Martio \cite[Theorem~C]{MalyM}, \cite[Theorem~1]{zap}, if $f\in W^{1,k}(\Sph^k,\R^m)$ is H\"older continuous, then it maps sets of measure zero to sets of $\mathcal{H}^k$-measure zero so
$\mathcal{H}^k(E)=0$ and hence the Hausdorff dimension of $f(\Sph^k)$ is at most $k$.
Let now $f$ be as in Corollary~\ref{T13}. Assume that $f$
is H\"older continuous.
According to the Fubini theorem for Sobolev functions (Lemma~\ref{T5}), $f$ restricted to almost all spheres $[n/2]$-dimensional spheres is a H\"older continuous map in $W^{1,[n/2]}$
so the image of almost every such sphere has Hausdorff dimension at most $[n/2]$ and hence its $\mathcal{H}^{[n/2]+1}$-measure is zero, so condition (a) from Theorem~\ref{main} is satisfied and the result follows. Similarly, if $f^{-1}$ is H\"older continuous, the condition (b) is satisfied and the result follows.
\end{proof}

\begin{proof}[Proof of Corollaries~\ref{T14} and~\ref{T15}]
Since $L^{m-1+\eps}_{\rm loc}\subset L^{m-1,1}_{\rm loc}$,
Corollary~\ref{T14} follows from Corollary~\ref{T15}. According to \cite[Theorem~C]{HKM}, mappings $f:\Sph^{k}\to\R^m$ with the weak derivative in $L^{k,1}$ map sets of measure zero to sets of $\mathcal{H}^k$-measure zero, so exactly the same argument as in the proof of Corollary~\ref{T13} yields that $\mathcal{H}^{k+1}(f(\Sph^k))=0$ and then, again as in the proof of Corollary~\ref{T13}, the result follows.
\end{proof}

The main idea in the proofs of Theorems~\ref{T1} and~\ref{T2} is to use the linking number.
If $n\geq 4$ and $p>[n/2]$, one can find linked spheres in $\Omega$ of dimensions less than $p$.
This allows one to use the Sobolev embedding theorem on the linked
spheres to control the topological linking number in terms of the Sobolev norm of the mapping. Since a sense preserving homeomorphism maps linked spheres onto linked topological spheres with the same linking number, one can use this fact to prove that the Jacobian of a sense preserving map cannot be negative on a set of positive measure. A similar argument is used when $n\leq 3$.

The proof of our Theorem~\ref{main} is based on a similar idea. However, we cannot use the Sobolev embedding theorem on spheres, because
now $p=k=[n/2]$ equals to the dimension of one of the linked spheres. This causes many technical problems and in order to handle them, we need to assume Sobolev regularity of the inverse map.

Although Theorem~\ref{main} gives an answer to Question~\ref{Q1} only in a very special case,
the main motivation behind Theorem~\ref{main} was to modify the technique of the linking number
so it could be used in the limiting case, in which we do not have the Sobolev embedding on spheres.
We believe that if the answer to Question~\ref{Q1} is in the negative, the proof should be based on
the linking number technique and we hope  that, with further modifications,
our new technique can lead
to the negative answer to Question~\ref{Q1} in full generality, for all $n\geq 4$.
However, we do not know yet how to do it and we are not even sure what the final answer to Question~\ref{Q1} is.

The paper is structured as follows. In Section~\ref{Prelim} we collect basic tools that are used in the proof of Theorem~\ref{main}. Some of the tools collected there are known, but some other are new and of independent interest. In Section~\ref{Proof1and2} we recall the proof of Theorems~\ref{T1} and~\ref{T2}. This helps to understand the main idea of our proof and to see what are the additional difficulties we have to face.
In the last Section~\ref{Proofof5} we prove Theorem~\ref{main}.
We put a lot of effort  to make the paper self-contained and accessible to those who are new to this area of research.

Notation in the paper is quite standard.
The Lebesgue measure of a set $A\subset\R^n$ is denoted by $|A|$.
By $\mathcal{H}^k$ we denote the $k$-dimensional Hausdorff measure. A $k$-dimensional open ball centered at a point $x$, with radius $r$, is denoted by $\bbbb^k(x,r)$, and $\bbbb^k$ denotes the open  unit ball in $k$ dimensions. Similarly, $\Sph^k$ denotes the unit $k$-dimensional sphere. The surface measure on $\Sph^k$ is denoted by $d\sigma(x)$.
Open half-space will be denoted by $\R^{n+1}_+=\R^n\times (0,\infty)$.
$W^{1,p}$ is the Sobolev space of functions $f\in L^p$ with $\nabla f\in L^p$. The Lorentz space is denoted by $L^{p,q}$. We do not recall the definition of this space since it does not play any role in our proofs. The Jacobian of a mapping $f\colon\R^n\to\R^n$ is denoted by $J_f=\det Df$.
A domain is an open and connected set. The integral average is denoted by
$$
\mvint_E f\, dx =\frac{1}{|E|}\int_E f\, dx.
$$
By $C$ we denote a generic constant whose value may change in a single string of estimates. Writing $C=C(n,m)$ we will indicate that the constant $C$ depends on $n$ and $m$ only.

\noindent
{\bf Acknowledgement.} We would like to thank Jan Mal\'y for providing us with a beautiful proof of Proposition~\ref{T10.5}.

A few days before completion of this work we learned the sad news that Professor Bogdan Bojarski had passed away. He was the PhD advisor of Piotr Haj\l{}asz and an inspiration for both of us. We mourn his passing, and we dedicate this paper with deep respect to his memory.

\section{Preliminaries}
\label{Prelim}
In this section we collect some basic facts that are used in the proof of the main result. We present the results in a slightly more general form than we actually need, because they might be useful for some other applications.

\subsection{Chain rule}
The main result of \cite{DHS} (see also \cite{Oliva}) provides an example of a surjective homeomorphism
$f:(0,1)^n\to (0,1)^n$, $n\geq 3$, such that $f\in W^{1,1}$, $f^{-1}\in W^{1,1}$ and $J_f=0$ a.e., $J_{f^{-1}}=0$ a.e. Note that $f^{-1}\circ f=\operatorname{Id}$, but the chain rule
$$
Df^{-1}(f(x))Df(x)=\operatorname{Id}
$$
cannot be satisfied on a set of positive measure because of the vanishing Jacobians.
In fact, $f$ maps the set of full measure to the set of measure zero where $Df^{-1}$ is not defined.
The situation is different if we assume that $J_f\neq 0$. Namely, we have
\begin{lemma}
\label{chain}
Let $U,V\subset\R^n$ be open sets. Assume that $f\in W^{1,1}_{\rm loc}(U,V)$, $g\in W^{1,1}_{\rm loc}(V,\R)$ and
$g\circ f\in W^{1,1}_{\rm loc}(U,\R)$. Then
\begin{equation}
\label{eq1}
D(g\circ f)(x)=Dg(f(x))Df(x)
\end{equation}
for almost all points $x$ in the set
$\{x\in U:\, J_f(x)\neq 0\}$.
\end{lemma}
\begin{remark}
In particular, the result says that $Dg(f(x))$ is well defined at almost all points $x$ such that
$J_f(x)\neq 0$.
\end{remark}
\begin{proof}
The set where the Jacobian is different than zero
splits into two sets where the Jacobian is positive and negative, respectively. Thus it suffices to show that \eqref{eq1} is satisfied almost everywhere in the set where the Jacobian is positive,
$$
X=\{x\in U:\, J_f(x)>0 \},
$$
because a similar argument can be applied to the set where the Jacobian is negative.

It is well known \cite{acerbif,bojarskih,hajlasz1} that $u\in W^{1,1}(\R^n)$ satisfies the pointwise inequality
\begin{equation}
\label{eq11}
|u(x)-u(y)|\leq C(n)|x-y|(\M |D u|(x)+\M |D u|(y))
\quad\text{a.e.,}
\end{equation}
where $\M |Du|(x)=\sup_{r>0}\mvint_{\bbbb^n(x,r)}|D u|(y)\, dy$ is the Hardy-Littlewood maximal function.
Hence for each $t>0$, $u$ restricted to the set $\{\M |Du|\leq t\}$ is Lipschitz continuous. This implies that $\R^n$ can be decomposed into a set
of measure zero and countably many sets such that on each of these sets $u$ is Lipschitz continuous. This fact and a partition of unity argument implies
that $U$ can be decomposed into Borel sets
\begin{equation}
\label{eq3}
U=N_o\cup\bigcup_{i=1}^\infty K_i
\end{equation}
such that $|N_o|=0$ and $f|_{K_i}$ is Lipschitz continuous.
We need to use here a partition of unity argument, because $f$ is defined in $U$, while \eqref{eq11}
applies to functions defined on $\R^n$.

It remains to show that \eqref{eq1} is satisfied at almost all points of the set $X\cap K_i$ for
each $i=1,2,\ldots$ Let $f_i$ be a Lipschitz extension of $f|_{X\cap K_i}$ to all of $\R^n$  (see \cite[Theorem~3.1]{EG}). According to the Rademacher theorem, \cite[Theorem~3.2]{EG}, $D f_i$ exists a.e. Also
$Df_i=Df$ a.e. in $X\cap K_i$. Indeed, $f-f_i=0$ in $X\cap K_i$ so $D(f-f_i)=0$ a.e. in $X\cap K_i$ by \cite[Theorem~4.4(iv)]{EG}. Let
$$
W_i=\{x\in X\cap K_i:\, \text{$Df_i(x)$ exists and $Df_i(x)=Df(x)$}\}.
$$
Since $|(X\cap K_i)\setminus W_i|=0$, it remains to show that \eqref{eq1} is satisfied at almost all points of the set $W_i$. Note that $J_{f_i}>0$ on $W_i$.
According to \cite[Lemma~3.3]{EG} we can decompose the set $W_i$ into a family of pairwise disjoint Borel sets
$$
W_i=\bigcup_{j=1}^\infty S_j
$$
such that $f_i|_{S_j}$ is bi-Lipschitz for each $j=1,2,\ldots$, and it remains to prove \eqref{eq1} at almost all points of $S_j$. If $|S_j|=0$, the result is
obvious, so we can assume that $|S_j|>0$ and hence $f(S_j)=f_i(S_j)$ has positive measure, too.

Since $g\in W^{1,1}_{\rm loc}$, we have a decomposition
$$
f(S_j)=M_o\cup\bigcup_{k=1}^\infty E_k,
\quad
|M_o|=0,
\quad
\text{$g|_{E_k}$ is Lipschitz continuous.}
$$
Indeed, we have a decomposition of $V$ similar to \eqref{eq3} and then we take intersections with the set $f(S_j)$.
We can also assume that $|E_k|>0$ for all $k$, as otherwise we could add sets $E_k$ of measure zero to the set $M_o$.
Since the mapping $f$ is bi-Lipschitz on $S_j$, we have that $|(f|_{S_j})^{-1}(M_o)|=0$ and it remains to prove that
\eqref{eq1} is satisfied at almost all points of $Z_{jk}=(f|_{S_j})^{-1}(E_k)$ for $k=1,2,\ldots$ Let $g_k$ be a Lipschitz extension of $g|_{E_k}$ to all of $\R^n$.
Then $g_k$ is differentiable a.e. and $Dg_k=Dg$ almost everywhere in $E_k$. Since $f|_{S_j}$ is bi-Lipschitz, the preimage $(f|_{S_j})^{-1}$ of the set of points in
$E_k$ where $g_k$ is not differentiable has measure zero. This and the classical chain rule for differentiable functions imply that $g\circ f=g_k\circ f_i$ in $Z_{jk}$ and
$$
D(g_k\circ f_i)(x)=Dg_k(f_i(x))Df_i(x)=Dg(f(x))Df(x)
\quad
\text{a.e. in $Z_{jk}$.}
$$
Since $g_k\circ f_i$ is Lipschitz continuous and it coincides with $g\circ f$ in $Z_{jk}$, it follows that
$D(g_k\circ f_i)=D(g\circ f)$ almost everywhere in $Z_{jk}$.
\end{proof}
As an immediate corollary we obtain the following result that will be used in the proof of Theorem~\ref{main}, see also \cite[Theorem~1.1 and~1.3]{DSS}, \cite[Lemma~2.1]{FMS}, \cite{HK}.
\begin{corollary}
\label{lem:inverse}
Assume that $\Omega\subset\R^n$ is a domain and
$f\in W^{1,1}_{\rm loc}(\Omega,\R^n)$ is a homeomorphism such that
${f^{-1}\in W^{1,1}_{\rm loc}(f(\Omega),\R^n)}$.
Then
\begin{equation}
\label{eq2}
(Df(x))^{-1}=Df^{-1}(f(x))
\end{equation}
almost everywhere in the set where $J_f\neq 0$.
\end{corollary}
In particular, Corollary~\ref{lem:inverse} applies to homeomorphisms described in Theorem~\ref{main}.

\begin{corollary}
\label{T6}
Let $U\subset\R^n$ be open and let $f\in W^{1,1}_{\rm loc}(U,\R^n)$ be continuous. If a compact set
$K\subset \{x\in U: J_f(x)\neq 0\}$ has positive measure, then the set $f(K)$ has positive measure.
\end{corollary}
\begin{remark}
In general, continuous mappings (even homeomorphisms) may map measurable sets to non-measurable sets. This is why we assume that $K$ is compact to guarantee measurability of the set $f(K)$.
\end{remark}
\begin{proof}
This is a corollary of the proof of Lemma~\ref{chain} and we assume the same notation as in the proof of Lemma~\ref{chain}. In particular we assume that the sets $W_i$ and $S_j$ are defined in the same way.

Let $K\subset\{x\in U:\, J_f(x)\neq 0\}$ be compact and of positive measure. Since $f$ is continuous, $f(K)$ is compact and hence measurable. One of the sets $K\cap\{ J_f>0\}$ or $K\cap \{J_f<0\}$ has positive measure. Without loss of generality we may assume that the set $K\cap\{J_f>0\}$ has positive measure.

The sets $W_i$ constructed in the proof of Lemma~\ref{chain} cover almost all points of the set $\{ J_f>0\}$ so
$|K\cap W_i|>0$ for some $i$. Since $W_i$ is the union of sets $S_j$, $|K\cap S_j|>0$ for some $j$. The mapping
$f|_{S_j}$ is bi-Lipschitz and it follows that $f(K\cap S_j)$ is measurable and of positive measure. Hence also $f(K)$ has positive measure.
\end{proof}

\subsection{Blow-up technique}
In this section we describe a {\em blow-up technique} (Lemma~\ref{T4})
that is often used in the study of partial differential equations. This technique has also been used in \cite{henclm1}.
Later, we generalize the blow-up technique to a simultaneous blow-up for a homeomorphism and its inverse, Lemma~\ref{T9}. This result will be used in the proof of Theorem~\ref{main}.

{\em All} results of this section are local in nature, so they are true for functions and mappings defined on domains in $\R^n$ and not necessarily on all of $\R^n$. However, for simplicity of notation we decided to formulate the results on $\R^n$.

We will need the following two classical lemmata, the first of which is due to Lebesgue.
\begin{lemma}[The Lebesgue Differentiation Theorem]
\label{lem:LDT}
If $f\in L^p_{\rm loc}(\R^n)$, $1\leq p<\infty$, then
\begin{equation}
\label{LDT}
\mvint_{\bbbb^n(x,r)} |f(y)-f(x)|^p dy\xrightarrow{r\to 0} 0
\quad
\text{for a.e. $x \in \R^n$.}
\end{equation}

\end{lemma}
The points $x\in \R^n$ where \eqref{LDT} holds true are called \emph{$p$-Lebesgue points} of $f$.

The second, due to Calder\'on and Zygmund, \cite[Theorem~12]{CZ}, is also an immediate consequence of \cite[Theorem 6.2]{EG}.
\begin{lemma}
\label{lem:CZ}
If $f\in W^{1,p}_{\rm loc}(\R^n)$, $1\le p<\infty$, then
\begin{equation}
\label{eq:CZ}
\mvint_{\bbbb^n(x,r)}\left|\frac{f(y)-f(x)-Df(x)\cdot(y-x)}{r}\right|^p\,dy\xrightarrow{r\to 0} 0
\quad
\text{for a.e. $x\in\R^n$.}
\end{equation}
\end{lemma}
Note that the above lemmata immediately generalize to the case of vector valued
functions $f\in W^{1,p}_{\rm loc}(\R^n,\R^k)$, since it suffices to apply them to components of $f$.
In particular we have that if $f\in W^{1,p}_{\rm loc}(\R^n,\R^k)$, then
\begin{equation}
\label{eq12}
\mvint_{\bbbb^n(x,r)} |Df(y)-Df(x)|^p dy\xrightarrow{r\to 0} 0
\quad
\text{for a.e. $x \in \R^n$.}
\end{equation}

\begin{definition}\label{D12}
Let $f\in W^{1,p}_{\rm loc}(\R^n,\R^k)$, $1\leq p<\infty$. We say that $x\in\R^n$ is a
{\em $p$-good point} for $f$ if
both of the integrals \eqref{eq:CZ} and \eqref{eq12} converge to zero.
\end{definition}
Clearly, almost all points of $\R^n$ are $p$-good points for $f\in W^{1,p}_{\rm loc}$.

The basic {\em blow-up technique} is described by the following lemma. It allows us to regard $f$ almost as a linear map near any $p$-good point.
\begin{lemma}
\label{T4}
Let $f\in W^{1,p}_{\rm loc}(\R^n,\R^k)$. For a $p$-good point $x_o\in\R^n$ and $r>0$ we define
$$
f_r(x)=\frac{f(x_o+rx)-f(x_o)}{r}
\quad
\text{and}
\quad
f_0(x)=Df(x_o)x.
$$
Then $f_r$ converges to the linear map $f_0$
in the norm of $W^{1,p}(\bbbb^n,\R^k)$ as $r\to 0$, where $\bbbb^n=\bbbb^n(0,1)$ is the unit ball.
\end{lemma}
\begin{proof}
Let $x_o\in\R^n$ be a $p$-good point for $f$.
Note that $Df_0(x)=Df(x_o)$. We have
$$
\mvint_{\bbbb^n}|f_r(x)-f_0(x)|^p\, dx =
\mvint_{\bbbb^n(x_o,r)}\left|\frac{f(y)-f(x_o)-Df(x_o)\cdot(y-x_o)}{r}\right|^p\, dy\to 0
$$
as $r\to 0$, and
$$
\mvint_{\bbbb^n}|Df_r(x)-Df_0(x)|^p\, dx =
\mvint_{\bbbb^n(x_o,r)} |Df(y)-Df(x_o)|^p\, dy\to 0
\quad
\text{as $r\to 0$.}
$$
\end{proof}
The rest of the section is devoted to a simultaneous blow-up for a homomorphism and its inverse.
\begin{lemma}
\label{T7}
Let $f\in W^{1,p}_{\rm loc}(\R^n,\R^n)$, $1\leq p<\infty$, be a homeomorphism such that
$f^{-1}\in W^{1,q}_{\rm loc}(\R^n,\R^n)$, $1\leq q<\infty$. Then almost all points of the set
$\{x\in\R^n:\, J_f(x)\neq 0\}$ have the following three properties satisfied simultaneously
\begin{enumerate}
\item[(a)] $x$ is a $p$-good point for $f$,
\item[(b)] $f(x)$ is a $q$-good point for $f^{-1}$,
\item[(c)] $(Df(x))^{-1}=(Df)^{-1}(f(x))$.
\end{enumerate}
\end{lemma}
\begin{proof}
A homeomorphic image of a Lebesgue measurable set need not be measurable, but a homeomorphic image of a Borel set is Borel, so we need to work with Borel sets.

Let $A\subset\R^n$ be a Borel set of $q$-good points for $f^{-1}$ such that $|\R^n\setminus A|=0$. Then
$f^{-1}(A)$ is Borel and hence measurable. Almost all points of the set
$\{J_f\neq 0\}\cap f^{-1}(A)$ have properties (a) and (b) and in order to show that almost all points of the set $\{J_f\neq 0\}$ have properties (a) and (b) it suffices to show that the set
$$
X=\{x\in\R^n:\, J_f(x)\neq 0\}\setminus f^{-1}(A)
$$
has measure zero. Suppose to the contrary that $|X|>0$. Let $K\subset X$ be a compact set of positive measure. Then
$f(K)\subset\R^n\setminus A$, and according to Corollary~\ref{T6}, $f(K)$ has positive measure. This is, however, impossible,
since $\R^n\setminus A$ has measure zero.

We proved that almost all points of the set $\{J_f\neq 0\}$ have properties (a) and (b). Now it follows from Corollary~\ref{lem:inverse} that almost
all points of the set $\{J_f\neq 0\}$ have all three properties (a), (b) and (c).
\end{proof}

The next lemma is easy to prove.
\begin{lemma}
\label{T8}
Let $f$ be as in Lemma~\ref{T7} and let $A\in GL(n)$ be a non-degenerate linear transformation on $\R^n$. If a point $x_o\in \{J_f\neq 0\}$
satisfies conditions (a), (b) and (c), then $x_o$ also satisfies conditions (a), (b), (c) for a homeomorphism $g=A\circ f$.
\end{lemma}
\begin{remark}
\label{R1}
Whether a point $x_o$ satisfies conditions (a), (b) and (c) for the mapping $f$ depends on the choice of representatives of $Df$ and $Df^{-1}$.
More precisely, it depends on how the values of $Df(x_o)$ and $Df^{-1}(f(x_o))$ are defined. However, we proved that no matter how we choose
representatives of $Df$ and $Df^{-1}$, almost all points will satisfy (a), (b), (c). If a point $x_o$ satisfies conditions (a), (b) and (c) for the
mapping $f$, then we will prove that
$x_o$ satisfies the same conditions for $g$, provided the representatives of $Dg$ and $Dg^{-1}$ are such that
$$
Dg(x_o)=ADf(x_o),
\quad
Dg^{-1}(g(x_o))=Df^{-1}(f(x_o))A^{-1},
$$
but we can make a choice of such representatives without any harm being done.
\end{remark}
\begin{proof}[Sketch of a proof]
The proof that $x_o$ is good for $g$ is straightforward. The proof that it is also good for $g^{-1}=f^{-1}\circ A^{-1}$ follows from the change of
variables $\Phi(y)=A^{-1}(y)$. Then the averages over the balls will became averages over scaled and translated ellipsoids $A^{-1}\bbbb^n$ and it
remains to observe (a well known fact) that if averages at \eqref{eq:CZ} and \eqref{eq12} (for $g^{-1}$ in place of $f$) over the balls $\bbbb^n$ converge
to zero, then also the averages over the ellipsoids  $A^{-1}\bbbb^n(g(x_o),r)$ converge to zero. Indeed, the average over an ellipsoid can be estimated from
above by the average over a larger ball that contains  $A^{-1}\bbbb^n(g(x_o),r)$, with a uniform constant that does not depend on the diameter of the ellipsoid.
Finally the condition (c) for $g=A\circ f$ is a consequence of linear algebra and the choice of representatives of $Dg$ and $Dg^{-1}$ (see Remark~\ref{R1}).
We leave details to the reader.
\end{proof}

Let $f$ and $x_o$ are as in Lemma~\ref{T8}. If $J_f(x_o)>0$ and
$A=Df(x_o)^{-1}$, then
$$
Dg(x_o)=Dg^{-1}(g(x_o))=\operatorname{Id}.
$$
If $J_f(x_o)<0$ and $A=\cR Df(x_o)^{-1}$, where
$$
\cR=\operatorname{Diag}(1,1,\ldots,1,-1)=
\begin{pmatrix}1&0&\hdots&0&0\\0&1&\hdots&0&0\\\vdots&\vdots&\ddots&\vdots&\vdots\\0&0&\hdots&1&0\\0&0&\hdots&0&-1\end{pmatrix},
$$
then
$$
Dg(x_o)=Dg^{-1}(g(x_o))=\cR,
$$
because $\cR=\cR^{-1}$.

In both cases the linear transformation $A$ has positive determinant.

The next lemma describes the simultaneous blow-up in the case of negative Jacobian.
This is what we will need in the proof of Theorem~\ref{main}.
In the case of positive Jacobian
one can easily formulate a similar result with $\mathcal{R}$ replaced by $\operatorname{Id}$.
\begin{lemma}
\label{T9}
Let $f\in W^{1,p}_{\rm loc}(\R^n,\R^n)$, $1\leq p<\infty$, be a homeomorphism such that
$f^{-1}\in W^{1,q}_{\rm loc}(\R^n,\R^n)$, $1\leq q<\infty$. Then for almost every point $x_o$ of the set
$\{x\in\R^n:\, J_f(x)<0\}$, there is a linear transformation $A$ with positive determinant such that
the homeomorphism $g=A\circ f$ satisfies
$$
\lim_{r\to 0^+}\Vert g_r-\mathcal{R}\Vert_{W^{1,p}(\bbbb^n,\R^n)}=
\lim_{r\to 0^+}\Vert (g_r)^{-1}-\mathcal{R}\Vert_{W^{1,q}(\bbbb^n,\R^n)}=0,
$$
where
$$
g_r(x)=\frac{g(x_o+rx)-g(x_o)}{r},
\quad
r>0.
$$
\end{lemma}
\begin{proof}
Almost every point of the set $\{ J_f<0\}$ satisfies conditions (a), (b) and (c) of Lemma~\ref{T7} for~$f$.
Fix such a point $x_o$. Then, by Lemma~\ref{T8}, $x_o$ satisfies conditions (a), (b) and (c) for $g=A\circ f$.
Choosing $A=\cR (Df(x_o))^{-1}$, we have that
$$
Dg(x_o)=Dg^{-1}(g(x_o))=\mathcal{R}.
$$
Since $J_f(x_o)<0$, it follows that $\det A>0$.
We identify $\mathcal{R}$ with the linear transformation $x\mapsto\mathcal{R}x$. Since
$x_o$ is a $p$-good point for $g$ and $g(x_o)$ is a $q$-good point for $g^{-1}$, Lemma~\ref{T4} implies that
$$
\lim_{r\to 0^+}\Vert g_r-\mathcal{R}\Vert_{W^{1,p}(\bbbb^n,\R^n)}=
\lim_{r\to 0^+}\Vert (g^{-1})_r-\mathcal{R}\Vert_{W^{1,q}(\bbbb^n,\R^n)}=0,
$$
where $(g^{-1})_r$ is the blow up of $g^{-1}$ at $g(x_o)$.
It remains to observe that $(g^{-1})_r=(g_r)^{-1}$, which easily follows from the definitions.
\end{proof}

\subsection{Fubini's theorem for Sobolev spaces}
The fact that almost all slices $f_x$, $f_{i,x}$ in the lemma below belong to the Sobolev space $W^{1,p}$ is well known. It is often called
Fubini's theorem for Sobolev spaces.
On the other hand, facts \eqref{eq4} and \eqref{eq5} are not so well known.
\begin{lemma}
\label{T5}
Let $f,f_i\in W^{1,p}\big((0,1)^n\big)$, $1\leq p<\infty$, with $f_i\to f$ as $i\to\infty$. Denote the points of
the cube $(0,1)^n$ by
$$
(x,y)\in(0,1)^{n-\ell}\times(0,1)^\ell=(0,1)^n
$$
and define
$$
f_x(y)=f(x,y),
\quad
f_{i,x}(y)=f_i(x,y).
$$
Then for almost all $x\in (0,1)^{n-\ell}$ we have $f_x,f_{i,x}\in W^{1,p}\big((0,1)^\ell\big)$ and there is a subsequence $f_{i_j}$ such that for almost all $x\in (0,1)^{n-\ell}$
\begin{equation}
\label{eq4}
\lim_{j\to\infty}\Vert f_{i_j,x}-f_x\Vert_{W^{1,p}((0,1)^\ell)}=0.
\end{equation}
Moreover, for any $\eps>0$, there is a compact set $K\subset (0,1)^{n-\ell}$ such that
$|(0,1)^{n-\ell}\setminus K|<\eps$ and
\begin{equation}
\label{eq5}
\lim_{j\to\infty}\sup_{x\in K}\Vert f_{i_j,x}-f_x\Vert_{W^{1,p}((0,1)^\ell)}= 0.
\end{equation}
\end{lemma}
\begin{proof}
The fact that $f_x\in W^{1,p}\big((0,1)^\ell\big)$ for almost all $x\in (0,1)^{n-\ell}$ is an easy consequence of the classical Fubini theorem
applied to a sequence of smooth functions approximating $f$ in $W^{1,p}\big((0,1)^n\big)$. We leave details to the reader. Similarly, we prove that for every $i\in\N$,
$f_{i,x}\in W^{1,p}\big((0,1)^\ell\big)$ for almost all $x\in (0,1)^{n-\ell}$.
Since we have countably many functions $f$, $f_i$, $i\in\N$,
there is a set $A\subset (0,1)^{n-\ell}$ of full measure such that $f_x,f_{i,x}\in W^{1,p}$
for all $x\in A$ and all $i\in\N$.
We have
\begin{equation*}
\begin{split}
&
\Vert f_i-f\Vert_{W^{1,p}((0,1)^n)}^p =\\
&
\int_{(0,1)^{n-\ell}}\int_{(0,1)^\ell} |f_i(x,y)-f(x,y)|^p +|Df_i(x,y)-Df(x,y)|^p\, dy\, dx\\
&=
\int_{(0,1)^{n-\ell}}\underbrace{\Vert f_{i,x}-f_x\Vert_{W^{1,p}((0,1)^\ell)}^p}_{F_i(x)}\, dx
\rightarrow 0
\quad
\text{as $i\to\infty$.}
\end{split}
\end{equation*}
In other words, $F_i\to 0$ in $L^1\big((0,1)^{n-\ell}\big)$. Therefore, there is a subsequence $F_{i_j}$ such that
$F_{i_j}(x)\to 0$ for almost all $x\in (0,1)^{n-\ell}$, which is \eqref{eq4}. Moreover, according to Egoroff's theorem
\cite[Theorem~1.16]{EG}, for any $\eps>0$ there is a compact set $K\subset (0,1)^{n-\ell}$ such that
$|(0,1)^{n-\ell}\setminus K|<\eps$ and $F_{i_j}\to 0$  uniformly on $K$, which is \eqref{eq5}.
\end{proof}
The above result allows for a lot of flexibility and instead of applying Fubini's theorem to the products of cubes, we can apply it for example to $\bbbb^{n-\ell}\times\Sph^{\ell}$, as described in the next result.

\begin{lemma}
\label{T10}
Let $f_r,f_0\in W^{1,p}({\bbbb}^{n-\ell}\times\Sph^\ell,\R^n)$, $1\leq p<\infty$,
$0<r<r_0$, be a family of mappings such that
$f_r\to f_0$ in $W^{1,p}$ as $r\to 0^+$.
If $r_i\searrow 0$ is a sequence decreasing to zero, then there is
a subsequence $r_{i_j}$ such that $f_j=f_{r_{i_j}}$ satisfies:

For almost all $x\in \bbbb^{n-\ell}$
$$
f_j|_{\{ x\}\times\Sph^\ell}\to f_0|_{\{ x\}\times\Sph^\ell}
\quad
\text{in $W^{1,p}(\{ x\}\times\Sph^\ell)$ as $j\to\infty$,}
$$
and for any $\eps>0$ there is a compact subset
of the unit ball $K\subset \bbbb^{n-\ell}$ such that
$|\bbbb^{n-\ell}\setminus K|<\eps$ and
$$
\sup_{x\in K}\Vert f_j-f_0\Vert_{W^{1,p}(\{ x\}\times\Sph^\ell)}\to 0
\quad
\text{as $j\to\infty$.}
$$
\end{lemma}

%\begin{lemma}
%Let $\iota:\overbar{\bbbb}^{n-\ell}\times\Sph^\ell\hookrightarrow\bbbb^n$ be a smooth embedding into the unit ball.
%We assume that the embedding is smooth up to the boundary in the sense that it extends to a smooth embedding of $\bbbb^{n-\ell}(0,1+\delta)\times\Sph^\ell$.

%Let $f_r,f_0\in W^{1,p}_{\rm loc}(\bbbb^n,\R^n)$, $1\leq p<\infty$,
%$0<r<\eps$, be a family of mappings such that
%$$
%f_r\to f_0
%\quad
%\text{in $W^{1,p}(\bbbb^n,\R^n)$ as $r\to 0$}.
%$$
%If $r_i\searrow 0$ is a sequence decreasing to zero, then there is
%a subsequence $r_{i_j}$ such that $f_j=f_{r_{i_j}}$ satisfies:

%For almost all $x\in \bbbb^{n-\ell}(0,1)$
%$$
%f_j|_{\iota(\{ x\}\times\Sph^\ell)}\to f_0|_{\iota(\{ x\}\times\Sph^\ell)}
%\quad
%\text{in $W^{1,p}(\iota(\{ x\}\times\Sph^\ell))$ as $j\to\infty$,}
%$$
%and for any $\eps>0$ there is a compact set
%of the unit ball $K\subset \bbbb^{n-\ell}$ such that
%$|\bbbb^{n-\ell}\setminus K|<\eps$ and
%$$
%\sup_{x\in K}\Vert f_j-f_0\Vert_{W^{1,p}(\iota(\{ x\}\times\Sph^\ell))}\to 0
%\quad
%\text{as $j\to\infty$.}
%$$
%\end{lemma}
\subsection{Traces and extensions}
The following lemma is known, but not very well known.
\begin{proposition}
\label{T10.5}
For $n\geq 1$ and $p>1$, there is a bounded linear extension operator
$$
E:W^{1,p}(\R^{n})\to W^{1,q}\cap C^\infty(\R^{n+1}_+),
\quad
\text{where $q=\frac{(n+1)p}{n}$.}
$$
In other words, $W^{1,p}(\R^{n})$ continuously embeds into the trace space
$W^{1-\frac{1}{q},q}(\R^{n})$ of $W^{1,q}(\R^{n+1}_+)$.
\end{proposition}
\begin{corollary}
\label{T11}
For $n\geq 2$, there is a bounded linear extension operator
$$
E:W^{1,n}(\R^{n})\to W^{1,n+1}\cap C^\infty(\R^{n+1}_+).
$$
In other words, $W^{1,n}(\R^{n})$ continuously embeds into the trace space
$W^{1-\frac{1}{n+1},n+1}(\R^{n})$ of $W^{1,n+1}(\R^{n+1}_+)$.
\end{corollary}
\begin{remark}
Corollary~\ref{T11} fails when $n=1$, see
\cite{BIN}, \cite[Exercise~14.36]{leoni} and \cite[Proposition~4]{SW}.
\end{remark}
\begin{remark}
Proposition~\ref{T10.5} was proved in \cite{BIN}. It also follows from Theorem 14.32,
Remark 14.35 and Proposition 14.40 in \cite{leoni} (first edition).
Corollary~\ref{T11} was also proved in \cite[Lemma~14]{GH} as a consequence of
Theorem 2.5.6, Theorem 2.7.1, Proposition 2.3.2.2(8), Theorem 2.5.7 and 2.5.7(9) in \cite{triebel}.
All of the arguments listed here are difficult.
Below we present an elementary and unpublished proof of Proposition~\ref{T10.5} due to Jan Mal\'y.
\end{remark}
\begin{proof}[Proof (due to Jan Mal\'y)]
In the proof we need the following result.
\begin{lemma}
\label{T12}
If $F\in L^1_{\rm loc}(\R^{n+1}_+)$, $f\in L^p(\R^{n})$, $p>1$, $n\geq 1$, and
\begin{equation}
\label{eq6}
|F(x,t)|\leq C\mvint_{\bbbb^{n}(x,t)}|f(y)|\, dy
\quad
\text{for $(x,t)\in\R^{n}\times (0,\infty)=\R^{n+1}_+$,}
\end{equation}
then
$$
\Vert F\Vert_{L^q(\R^{n+1}_+)}\leq C\Vert f\Vert_{L^p(\R^{n})},
\quad
\text{where $q=\frac{(n+1)p}{n}$.}
$$
\end{lemma}
\begin{proof}
Since the right-hand side of \eqref{eq6} is bounded, up to the constant
$C$, by the maximal function $\M f(x)$, we have
\begin{equation}
\label{eq7}
\int_0^r |F(x,t)|^q\, dt \leq
Cr(\M f)^q(x)
\quad
\text{for $r>0$.}
\end{equation}
On the other hand, the inequality
$$
|F(x,t)|\leq C\left(\,\mvint_{\bbbb^{n}(x,t)}|f(y)|^p\, dy\right)^{1/p}\leq
C t^{-\frac{n}{p}} \Vert f\Vert_p
$$
yields
\begin{equation}
\label{eq8}
\int_r^\infty |F(x,t)|^q\, dt\leq C\Vert f\Vert_p^q r^{1-\frac{nq}{p}}=
C\Vert f\Vert_p^q r^{-n}
\quad
\text{for $r>0$.}
\end{equation}
We used here the fact that
$$
\frac{nq}{p}=n+1>1.
$$
Now if we choose $r>0$ such that
$$
r(\M f)^q(x)=\Vert f\Vert_p^qr^{-n},
\quad
\text{i.e.,}
\quad
r=\left(\frac{\Vert f\Vert_p}{\M f(x)}\right)^{\frac{p}{n}},
$$
then the right-hand sides of \eqref{eq7} and~\eqref{eq8}  are equal to
$C\Vert f\Vert_p^{p/n}(\M f(x))^p$. Adding integrals at \eqref{eq7} and~\eqref{eq8} yields
$$
\int_0^\infty |F(x,t)|^q\, dt \leq C\Vert f\Vert_p^{p/n}(\M f(x))^p
$$
and Fubini's theorem along with boundedness of the maximal operator in~$L^p$, $p>1$, give
$$
\int_{\R^{n+1}_+} |F(x,t)|^q\, dt\, dx\leq
C\Vert f\Vert_p^{p/n}\Vert f\Vert_p^p=C\Vert f\Vert_p^q.
$$
\end{proof}
Now we can complete the proof of Proposition~\ref{T10.5}.

Let $\varphi\in C_0^\infty(\bbbb^{n})$, $\varphi\geq 0$, $\int_{\bbbb^{n}}\varphi(x)\, dx =1$, and
$\varphi_t(x)=t^{-n}\varphi(x/t)$.

For $f\in L^1_{\rm loc}(\R^{n})$
we define
$$
(Ef)(x,t)=(f*\varphi_t)(x)=\int_{\bbbb^{n}}f(x-ty)\varphi(y)\, dy,
\quad
(x,t)\in\R^{n+1}_+.
$$
Clearly, properties of the convolution guarantee that $Ef\in C^\infty(\R^{n+1}_+)$ and a simple change of
variables yields
$$
|(Ef)(x,t)|=\left|\,\int_{\bbbb^{n}(x,t)}f(y)\varphi_t(x-y)\, dy\right|\leq C \mvint_{\bbbb^{n}(x,t)}|f(y)|\, dy,
$$
where the constant $C$ depends on $\varphi$ only. Therefore, if $f\in L^{p}(\R^{n})$,
Lemma~\ref{T12} gives the estimate
\begin{equation}
\label{eq9}
\Vert Ef\Vert_{L^q(\R^{n+1}_+)}\leq C\Vert f\Vert_{L^{p}(\R^{n})}.
\end{equation}
Assume now that $u\in W^{1,p}(\R^{n})$. We have
\begin{equation}
\label{eq10}
|\nabla (Eu)(x,t)|=|\nabla_{(x,t)}(Eu)(x,t)|\leq
C\mvint_{\bbbb^{n}(x,t)}|\nabla u(y)|\, dy.
\end{equation}
Indeed,
$$
|\nabla_{x}(Eu)(x,t)|=\left|\,\int_{\bbbb^{n}}(\nabla u)(x-ty)\varphi(y)\, dy\right|\leq
C\mvint_{\bbbb^{n}(x,t)}|\nabla u(y)|\, dy,
$$
and
$$
\left|\frac{\partial}{\partial t}(Eu)(x,t)\right| =
\left|\,\int_{\bbbb^{n}}(\nabla u)(x-ty)\cdot(-y)\varphi(y)\, dy\right|\leq
C\mvint_{\bbbb^{n}(x,t)}|\nabla u(y)|\, dy,
$$
because $|-y|\leq 1$ for $y\in\bbbb^{n}$. Now \eqref{eq10} and Lemma~\ref{T12} imply that
$$
\Vert\nabla (Eu)\Vert_{L^q(\R^{n+1}_+)}\leq C\Vert\nabla u\Vert_{L^{p}(\R^{n}).}
$$
This estimate and \eqref{eq9} for $f=u$ complete the proof.
\end{proof}

A localized version of Corollary~\ref{T11} gives
the following result:
\begin{lemma}
\label{L26}
Let $n\geq 2$. Then there is an extension operator
$$
E:W^{1,n}(\partial(\Sph^n\times [0,1]))\to
W^{1,n+1}\cap C^\infty (\Sph^n\times(0,1))
$$
such that if
$h=f_0$ on $\Sph^n\times\{0\}$ and $h=f_1$ on $\Sph^n\times\{1\}$,
then traces (in the Sobolev sense) satisfy $\Tr Eh(\cdot,0)=f_0$, $\Tr Eh(\cdot,1)=f_1$,
and
$$
\Vert Eh\Vert_{W^{1,n+1}(\Sph^n\times(0,1))}\leq
C(n)\left(\Vert f_0\Vert_{W^{1,n}(\Sph^n)}+\Vert f_1\Vert_{W^{1,n}(\Sph^n)}\right).
$$
\end{lemma}
The next result seems new. Its proof uses some ideas from
the proof of  \cite[Proposition~3.3]{hajlaszm}.
\begin{proposition}
\label{L27}
Let $n\geq 2$, $m\geq n+1$ and let $f\in W^{1,n}\cap C^0(\Sph^n,\R^m)$, $g\in C^\infty(\Sph^n,\R^m)$. Then there is a function
$H\in C^0(\Sph^n\times [0,1],\R^m)\cap C^\infty(\Sph^n\times [0,1))$
such that $H(x,0)=g(x)$, $H(x,1)=f(x)$ and
$$
\mathcal{H}^{n+1}(H(\Sph^n\times [0,1)))\leq
C\Vert f-g\Vert_{W^{1,n}(\Sph^n)}
\left(\Vert f-g\Vert_{W^{1,n}(\Sph^n)}+\Vert Dg\Vert_{L^{n+1}(\Sph^n)}\right)^n,
$$
where the constant $C$ depends on $n$ and $m$ only
and $\cH^{n+1}$ denotes the Hausdorff measure.
\end{proposition}
\begin{remark}
\label{rem1}
Fix $g\in C^\infty(\Sph^n,\R^m)$.
Let $f,f_k\in W^{1,n}\cap C^0(\Sph^n,\R^m)$
and let $H,H_k$ be the homotopies for $f$ and $f_k$ constructed as in Proposition~\ref{L27}.
If $f_k\to f$ uniformly on $\Sph^n$ as $k\to\infty$
(we do not require convergence in the Sobolev norm), then
$H_k\to H$ uniformly on $\Sph^n\times [0,1]$. Indeed, the homotopies are defined by the formula
\eqref{eq15} and the extension operator $Eh$ is defined through the averaging and multiplication by a cut-off function, and such a construction is continuous in the uniform norm.
\end{remark}
\begin{remark}
The proposition has a clear geometric interpretation. The image of a continuous mapping
$f\in W^{1,n}\cap C^0(\Sph^n,\R^m)$ can be very large. It can even fill a ball in $\R^m$. However, if $f$ is very close in the $W^{1,n}$ norm to a fixed smooth map $g$,
then there is a homotopy between $f$ and $g$ such that the $\mathcal{H}^{n+1}$-volume of the
image of the homotopy, except the endpoint where the homotopy equals $f$, is very small. We hope that this result might be useful for other applications.
\end{remark}
In the proof we will need the following estimate.
\begin{lemma}
\label{L28}
If $A$ and $B$ are two $n\times n$-matrices, then
\begin{equation}\label{eq:l28}
|\det (A+B)|\leq C(n)(|A|^n+|B|^n),
\end{equation}
where $|\cdot |$ stands for the Hilbert-Schmidt norm of a matrix.
\end{lemma}
\begin{proof}
Since the determinant is continuous and homogeneous of order $n$,
we immediately get that
$$
|\det A|=|A|^n\left|\det\left(\frac{A}{|A|}\right)\right|\leq\Lambda |A|^n,
\quad
\text{where}
\quad
\Lambda=\sup\{ |\det M|:\, |M|\leq 1\}.
$$
Then \eqref{eq:l28} follows from the triangle inequality and the standard convexity estimate $(a+b)^n\leq 2^{n-1}(a^n+b^n)$ for $a,b\geq 0$.
\end{proof}

\begin{proof}[Proof of Proposition~\ref{L27}]
Let
$$
h(x)=
\begin{cases}
f-g & \text{on $\Sph^n\times\{ 1\}$},\\
0   & \text{on $\Sph^n\times\{ 0\}$},
\end{cases}
$$
and define
\begin{equation}
\label{eq15}
H(x,t)=(Eh)(x,t)+g(x)
\quad
\text{for $(x,t)\in\Sph^n\times [0,1]$,}
\end{equation}
where $E$ is the extension operator from Lemma~\ref{L26}.
Since the extension $Eh$ is continuous (smooth) up to the boundary if the function on the boundary is continuous (smooth), we conclude that
$H\in C^0(\Sph^n\times [0,1],\R^m)\cap C^\infty(\Sph^n\times [0,1))$
and $H(x,0)=g(x)$, $H(x,1)=f(x)$.

According to the area formula \cite{EG} we have that
\begin{equation}
\label{eq13}
\mathcal{H}^{n+1}(H(\Sph^n\times [0,1)))=\int_0^1\int_{\Sph^n} J_H(x,t)\, d\sigma(x)\, dt,
\end{equation}
where
$$
J_H(x,t)=\sqrt{\det\big((DH)^T(DH)\big)}.
$$
Denote the derivative in $\Sph^n\times [0,1]$ by $D=(D_x,\partial_t)$, where $D_x$
is the derivative on $\Sph^n$. Then
$$
DH(x,t)=(D_x(Eh)+D_xg,\partial_t(Eh)).
$$
The Cauchy-Binet formula \cite[Sect. 3.2.1, Theorem 4]{EG},
the Laplace expansion along the last column $(\partial_t(Eh))_I$, and Lemma~\ref{L28}
yield the following estimate, where the sum is taken over
all $I=(i_1,\ldots,i_{n+1})$, $1\leq i_1<\ldots< i_{n+1}\leq m$:
$$
J_H
=
\sqrt{\sum_{I}\Big(\det\big((D_x(Eh))_I+(D_xg)_I, (\partial_t(Eh))_I\big)\Big)^2}\\
\leq
C |\partial_t(Eh)|(|D_x(Eh)|^n+|D_x g|^n).
$$
Therefore \eqref{eq13} gives
\begin{align*}
\mathcal{H}^{n+1}(H(\Sph^n&\times [0,1)))
\leq
C\Vert Eh\Vert_{W^{1,n+1}(\Sph^n\times(0,1))}^{n+1}\\
&+
C\left(\int_0^1\int_{\Sph^n} |D_x g|^{n+1}\right)^{\frac{n}{n+1}}
\left(\int_0^1\int_{\Sph^n}|\partial_t(Eh)|^{n+1}\right)^\frac{1}{n+1}\\
&\leq C\Vert f-g\Vert^{n+1}_{W^{1,n}(\Sph^n)}+
\Vert Dg\Vert_{L^{n+1}(\Sph^n)}^n\Vert f-g\Vert_{W^{1,n}(\Sph^n)}.
\end{align*}
\end{proof}

\subsection{The local degree and the linking number}
\label{Sec2.5}
To keep the paper self contained, we present here a short introduction to the theory of local degree. The standard references are the books of Fonseca and Gangbo \cite{FG} and Outerelo and Ruiz \cite{OR}. Then, at the end of the section, we discuss the linking number.

Throughout this section, we assume that $\Omega\subset\R^n$ is a domain. By $C^1(\overbar{\Omega},\R^n)$ we denote the set of all these mappings from $\overbar{\Omega}$ to $\R^n$ which admit an extension to a $C^1$ mapping on some open $U\supset\overbar{\Omega}$.

We begin by defining the local degree for $C^1$ mappings at their regular values.
\begin{definition}[\mbox{\cite[Definition 1.2]{FG}}] \label{defn12}
Assume $\phi\in C^1(\overbar{\Omega},\R^n)$. Let $p\in \R^n$ be a regular value of $\phi$ and $p\not\in \phi(\partial \Omega)$. We define the local degree of $\phi$ at $p$ with respect to $\Omega$ as
$$
\deg(\phi,\Omega,p)=\sum_{x\in \phi^{-1}(p)} \sgn J_\phi(x).
$$
Moreover, if $p\not \in \phi(\overbar{\Omega})$, we set $\deg(\phi,\Omega,p)=0$.
\end{definition}
It turns out that $\deg(\phi,\Omega,\cdot)$ is constant on connected components of ${\R^n\setminus \phi(\partial\Omega)}$.
\begin{proposition}[\mbox{\cite[Proposition 1.8]{FG}}]
\label{prop18}
Let $V$ be a connected component of $\R^n\setminus \phi(\partial\Omega)$ and assume $p_1,p_2\in V$ are regular values of $\phi$. Then $\deg(\phi,\Omega,p_1)=\deg(\phi,\Omega,p_2)$.
\end{proposition}
Proposition \ref{prop18} allows us to define the local degree also at critical points of $\phi$, as long as they are not in the image of $\partial\Omega$.
\begin{definition}[\mbox{\cite[Definition 1.9]{FG}}]
Assume $\Omega$ and $\phi$ are as in Definition \ref{defn12} and let $p\in \phi(\Omega)\setminus \phi(\partial \Omega)$ be a critical value of $\phi$. Let $p_1$ be any regular value of $\phi$ such that $p$ and $p_1$ lie in the same connected component of $\R^n\setminus \partial\Omega$. We set $\deg(\phi,\Omega,p)=\deg(\phi,\Omega,p_1)$.
\end{definition}
Note that by Sard's Lemma such $p_1$ always exists; Proposition \ref{prop18} shows that the above definition does not depend on the choice of $p_1$.

The local degree is a $C^1$ homotopy invariant:
\begin{proposition}[\mbox{\cite[Theorem 1.12]{FG}}]
\label{thm112}
If $H:\overbar{\Omega}\times[0,1]\to\R^n$ is a $C^1$ mapping such that $H(\cdot,0)=\phi(\cdot)$, $H(\cdot,1)=\psi(\cdot)$
and $p\not\in H(\partial\Omega\times[0,1])$,
then $\deg(\phi,\Omega,p)=\deg(\psi,\Omega,p)$.
\end{proposition}
Proposition \ref{thm112} allows us to extend the notion of local degree to continuous mappings:
\begin{definition}[\mbox{\cite[Definition 1.18]{FG}}]\label{defn118}
If $\phi\in C(\overbar{\Omega},\R^n)$ and $p\not \in \phi(\partial\Omega)$, we set  $\deg(\phi,\Omega,p)=\deg(\psi,\Omega,p)$, where $\psi$ is any mapping in $ C^1(\overbar{\Omega},\R^n)$  such that $\sup_{x\in \overbar{\Omega}} |\phi(x)-\psi(x)|<\mathrm{dist}(p,\phi(\partial \Omega))$.
\end{definition}
One easily checks that the above definition is independent on the choice of $\psi$: if we choose $\psi_1, \psi_2\in C^1(\overbar{\Omega},\R^n)$ satisfying $\sup_{x\in\overbar{\Omega}}|\phi(x)-\psi_i(x)|<\mathrm{dist}(p,\phi(\partial \Omega))$, $i=1,2$, then $p\not\in H(\partial \Omega\times[0,1])$, where $H$ is the standard homotopy between $\psi_1$ and $\psi_2$, $H(x,t)=(1-t)\psi_1(x)+t\psi_2(x)$, and thus, by Proposition \ref{thm112}, the local degrees of $\psi_1$ and $\psi_2$ at $p$ are the same.

\begin{remark}
\label{R2}
In fact, a standard smoothing argument shows that the local degree for continuous maps at a point $p$, given by the above definition, is a homotopy invariant (without the $C^1$ assumption), as long as the homotopy $H$ does not map any points of $\partial \Omega$ into $p$, i.e. $p\not\in H(\partial \Omega\times [0,1])$.
\end{remark}

We shall need the following deep facts on the local degree.
\begin{proposition}[\mbox{Multiplication theorem, \cite[Theorem 2.10]{FG}}]
\label{propo210}
Assume $\Omega\subset\R^n$ is a domain, $\phi\in C(\overbar{\Omega},\R^n)$, $V\subset \R^n$ is a domain containing $\phi(\overbar{\Omega})$ and $\psi\in C(\overbar{V},\R^n)$. Let $D=V\setminus \phi(\partial\Omega)$ and denote by $D_i$ the connected components of $D$. Then for any $p\not\in (\psi\circ\phi)(\partial \Omega)\cup \psi(\partial V)$ we have $p\not\in \psi(\partial D_i)$, and the following formula holds:
\begin{equation}
\label{pro39f}
\deg(\psi\circ\phi,\Omega,p)=\sum_{i}\deg(\psi,D_i,p)\deg(\phi,\Omega,q_i)
\end{equation}
for arbitrary $q_i\in D_i$.
\end{proposition}
%The next result follows from Brouwer's Invariance of Domain Theorem \cite[Theorem 3.30]{FG} and Proposition \ref{propo210}, applied to $h\circ h^{-1}=\mathrm{id}$.
\begin{proposition}
\label{propo40}
If $\psi$ and $\phi$ are as in Proposition \ref{propo210} and additionally $\psi$ and $\phi$ are homeomorphisms, then
\begin{itemize}
\item[a)] for any $p\not\in (\psi\circ\phi)(\partial \Omega)\cup \psi(\partial V)$
\begin{equation}\label{pro40f}
\deg(\psi\circ\phi,\Omega,p)=\deg(\psi,\phi(\Omega),p)\deg(\phi,\Omega,\psi^{-1}(p)),
\end{equation}
\item[b)] for every $q\in \phi(\Omega)$ we have either
$$
\deg(\phi,\Omega,q)=\deg(\phi^{-1},\phi(\Omega),\phi^{-1}(q))=1
$$ or
$$
\deg(\phi,\Omega,q)=\deg(\phi^{-1},\phi(\Omega),\phi^{-1}(q))=-1.
$$
\end{itemize}
\end{proposition}
\begin{proof}
If $p\not\in \psi(\phi(\Omega))$, then both sides of \eqref{pro40f} are zero. Assume thus $p\in \psi(\phi(\Omega))$.

By Brouwer's Invariance of Domain Theorem \cite[Theorem 3.30]{FG}, $\phi(\partial\Omega)=\partial\phi(\Omega)$ and there is only one component $D_i$ of $V\setminus \phi(\partial\Omega)$ such that $\psi(D_i)$ contains $p$, namely $D_i=\phi(\Omega)$, thus all the terms of the sum in \eqref{pro39f} are zero, except the one with  $D_i=\phi(\Omega)$. Also, we may choose $q_i=\psi^{-1}(p)$, which gives \eqref{pro40f}.

To prove b), take any domain $\Omega'$ such that $q\in \phi(\Omega')$ and $\overbar{\Omega}'\subset \Omega$. Applying a) to $\phi|_{\overbar{\Omega}'}:\overbar{\Omega}'\to\R^n$ and $\psi=\phi^{-1}:\phi(\overbar{\Omega})\to\R^n$, we see that
$$
1=\deg(\Id,\Omega',\phi^{-1}(q))=\deg(\phi^{-1},\phi(\Omega'),\phi^{-1}(q))\deg(\phi,\Omega',q),
$$
thus $\deg(\phi,\Omega',q)=\deg(\phi^{-1},\phi(\Omega'),\phi^{-1}(q))=\pm 1$. What remains to prove is that $\deg(\phi,\Omega',q)=\deg(\phi,\Omega,q)$ and $\deg(\phi^{-1},\phi(\Omega'),\phi^{-1}(q))=\deg(\phi^{-1},\phi(\Omega),\phi^{-1}(q))$.
Let $\zeta\in C^1(\overbar{\Omega},\R^n)$ be such that
$$
\sup_{\overbar{\Omega}}|\phi-\zeta|<\dist(q,\phi(\partial\Omega'))<\dist(q,\phi(\partial\Omega)).
$$
Then, by Definition \ref{defn118}, $\deg(\phi,\Omega,q)=\deg(\zeta,\Omega,q)$ and $\deg(\phi,\Omega',q)=\deg(\zeta,\Omega',q)$. However, for any $z\in \Omega\setminus\Omega'$,
$$
|\zeta(z)-q|\geq
|q-\phi(z)|-|\phi(z)-\zeta(z)|>
\dist(q,\phi(\partial\Omega'))-\dist(q,\phi(\partial\Omega'))=0,
$$
thus $q\not\in \zeta(\Omega\setminus \Omega')$ and applying Definition \ref{defn12} we see that $\deg(\zeta,\Omega,q)=\deg(\zeta,\Omega',q)$.
Calculations for $\phi^{-1}$ follow the same steps. This concludes the proof of b).
\end{proof}
A homeomorphism $h:\Omega\to\R^n$ with degree $+1$ at all its values is called \emph{sense-} or \emph{orientation-preserving}; if the degree is $-1$ at all values, we call it \emph{sense-reversing}.

As an immediate corollary of Proposition \ref{propo40}, b) and Proposition \ref{prop18}, we obtain that every homeomorphism of a domain is either sense-preserving or sense-reversing.
\begin{corollary}[\mbox{c.f. \cite[Theorem 3.35]{FG}}]
Assume $\Omega\subset\R^n$ is open and connected and $h:\overbar{\Omega}\to\R^n$ is a homeomorphism. Then either for every ${p\in h(\Omega)}$ we have ${\deg(h,\Omega,p)=1}$ or for every ${p\in h(\Omega)}$ we have ${\deg(h,\Omega,p)=-1}$.
\end{corollary}
\begin{corollary}
If $h:\Omega\to\R^n$ is a sense-preserving (reversing) homeomorphism and
$\Omega'\subset\Omega$, then $h|_{\Omega'}:\Omega'\to\R^n$ is also sense-preserving (reversing).
\end{corollary}
This is a corollary from the proof of Proposition~\ref{propo40}, where
we showed in the last step that $\deg(\phi,\Omega,q)=\deg(\phi,\Omega',q)$.

The terms \emph{sense-preserving} and \emph{sense-reversing} are justified by the following fact.
\begin{proposition}
\label{propo31}
Assume $\Omega\subset\R^n$ is a domain, $\phi\in C(\overbar{\Omega},\R^n)$ and $h:\overbar{U}\to\R^n$ is a homeomorphism of a domain $U\supset \phi(\overbar{\Omega})$.
\begin{itemize}
\item If $h$ is sense-preserving, then for every
$p\in U\setminus \phi(\partial\Omega)$ we have
$$
\deg(\phi,\Omega ,p)=\deg(h\circ\phi,\Omega,h(p)).
$$
\item If $h$ is sense-reversing, then for every $p\in U\setminus \phi(\partial\Omega)$ we have
$$
\deg(\phi,\Omega ,p)={-\deg(h\circ\phi,\Omega,h(p))}.
$$
\end{itemize}
\end{proposition}
\begin{proof}
Assume $h$ is sense-preserving and apply Proposition \ref{propo210} to $h\circ \phi$: let $D=U\setminus \phi(\partial \Omega)$ and denote by $D_i$ the connected components of $D$. Then $p\in D_j$ for exactly one $j\in\bbbn$ and we have for any $q_i\in D_i$, $i\neq j$,
\begin{align*}
\deg(h\circ\phi,\Omega,h(p))
&=\deg(h,D_j,h(p))\deg(\phi,\Omega,p)+
\sum_{i\neq j} \deg(h,D_i,h(p))\deg(\phi,\Omega,q_i)\\
&=\deg(\phi,\Omega,p),
\end{align*}
because $\deg(h,D_j,h(p))=+1$, while $\deg(h,D_i,h(p))=0$ for $i\neq j$,
since $h(p)\not\in h(D_i)$.

The case of sense-reversing $h$ is proved in exactly the same way.
\end{proof}

\begin{definition}
If $M$ and $N$ are compact, connected, oriented, smooth $n$-dimensional manifolds without boundary and  $\phi:M\to N$ is $C^\infty$ smooth, then we define
$$
\deg(\phi,y)=
\sum_{x\in\phi^{-1}(y)}\sgn J_\phi(x),
$$
where $y\in N$ is a regular value of $\phi$, and $J_\phi(x)$ is the determinant of the derivative $D\phi(x):T_xM\to T_{\phi(x)}N$. It turns out that $\deg(\phi,y)$ does not depend on the choice of a regular value $y$. The common value of all $\deg(\phi,y)$ is denoted by
$\deg\phi$ and is called the {\em degree of $\phi$}.
\end{definition}
One can prove that homotopic mappings have equal degrees. Since every continuous mapping is homotopic to a smooth one, one can extend the notion of degree to the class of all continuous mappings $\phi:M\to N$.
For more details, see \cite{Milnor}.

The following result relates the local degree of $\phi$ with the topological degree of $\phi$ restricted to the boundary.
\begin{proposition}[\mbox{\cite[Proposition IV.4.6]{OR}}]
\label{propo46}
Let $\partial \Omega$ be a connected, compact and smooth manifold oriented by the outward normal vector (\cite[Section~II.7.7]{OR})
and  assume $\phi\in C(\overbar{\Omega},\R^n)$ is such that $\phi|_{\partial\Omega}:\partial\Omega\to \Sph^{n-1}$. Then
$$
\deg \phi|_{\partial \Omega}=\deg(\phi,\Omega,0).
$$
\end{proposition}

Our main purpose for introducing the local degree is to justify the properties of yet another invariant, the linking number.

The linking number is an important and well studied invariant in the theory of knots. It was introduced by Gauss in a short note of 1833 \cite{Gauss} (see also \cite{ricca} for a nice historical account and modern interpretation):
if $\gamma_1$, $\gamma_2$ are two parameterized, non-intersecting, oriented curves in $\R^3$, $\gamma_1,~\gamma_2:\Sph^1\to\R^3$, then the linking number $\ell(\gamma_1,\gamma_2)$ is defined (in modern notation) as the integral
\begin{equation}\label{Gauss int}
\ell(\gamma_1,\gamma_2)=\frac{1}{4\pi}\int_{\Sph^1\times\Sph^1}\frac{\det(\gamma_1'(s),\gamma_2'(t),\gamma_1(s)-\gamma_2(t))}{|\gamma_1(s)-\gamma_2(t)|^3}\,ds\, dt.
\end{equation}
The Gauss map for $\gamma_1$, $\gamma_2$ is defined as
\begin{equation}\label{Gauss map}
\Sph^1\times\Sph^1\ni(s,t)\xmapsto{~~\psi~~}\frac{\gamma_1(s)-\gamma_2(t)}{|\gamma_1(s)-\gamma_2(t)|}\in \Sph^2.
\end{equation}
It turns out (see e.g. \cite{ricca}) that the linking number given by \eqref{Gauss int} is equal to the degree of the Gauss map: $\ell(\gamma_1,\gamma_2)=\deg(\psi)$.
In colloquial terms, the linking number tells us how many times (counting directions) one curve winds around the other.

These definitions have been later generalized to pairs of non-intersecting manifolds in higher dimensions (see the paper by M. Kervaire, \cite{Kervaire}, who attributes the idea to A. Shapiro). Here, we shall restrict ourselves to the case when the manifolds in question are two oriented spheres $\Sph^k$, $\Sph^l$, where $k+l=n-1$, continuously embedded in $\R^n$ in such a way that the embedded spheres do not intersect. Then, if we denote the embeddings as before, in the case of curves, by $\gamma_1:\Sph^k\to\R^n$ and $\gamma_2:\Sph^l\to\R^n$, we can define the Gauss map of the two embeddings by the formula \eqref{Gauss map}, (this time, however, $\psi:\Sph^k\times\Sph^l\to \Sph^{n-1}$), and the linking number of the two embedded (oriented) spheres, identified here with the embedding maps $\gamma_1$ and $\gamma_2$, is defined again as $\ell(\gamma_1,\gamma_2)=\deg \psi$.

More generally, we can define, by the same formula, the linking number $\ell(\gamma_1,\gamma_2)$ between any two continuous maps
$\gamma_1:\Sph^k\to\R^n$ and $\gamma_2:\Sph^l\to\R^n$, $k+l=n-1$, provided
$\gamma_1(\Sph^k)\cap\gamma_2(\Sph^l)=\varnothing$.

The linking number is a homotopy invariant in the sense that if $\gamma_1, \tilde{\gamma}_1:\Sph^k\to\R^n\setminus \gamma_2(\Sph^l)$ are two mappings of $\Sph^k$, which are homotopic in $\R^n\setminus \gamma_2(\Sph^l)$ (i.e. neither the two images $\gamma_1(\Sph^k)$, $\tilde{\gamma}_1(\Sph^k)$, nor the image of the homotopy between them intersects $\gamma_2(\Sph^l)$), then $\ell(\gamma_1,\gamma_2)=\ell(\tilde{\gamma}_1,\gamma_2)$. Indeed, if $\Gamma:\Sph^k\times [0,1]\to \R^n$ is the homotopy between $\gamma_1$ and $\tilde{\gamma}_1$, the image of which is disjoint with the image of $\gamma_2$, then
$\Psi:\Sph^k\times \Sph^l\times[0,1]\to\Sph^n$ given by the formula
$$
(s,t,r)\xmapsto{\Psi}\frac{\Gamma(s,r)-\gamma_2(t)}{|\Gamma(s,r)-\gamma_2(t)|}
$$
is a homotopy between
$$
\psi(s,t)=\frac{\gamma_1(s,t)-\gamma_2(s,t)}{|\gamma_1(s,t)-\gamma_2(s,t)|}
\quad
\text{and}
\quad
\tilde{\psi}(s,t)=\frac{\tilde{\gamma_1}(s,t)-\gamma_2(s,t)}{|\tilde{\gamma_1}(s,t)-\gamma_2(s,t)|},
$$
and thus the linking numbers $\ell(\gamma_1,\gamma_2)=\deg \psi$ and
$\ell(\tilde{\gamma_1},\gamma_2)=\deg \tilde{\psi}$ are the same.

The following known invariance result is very hard to find in the literature, thus we present the proof here (essentially the same argument, in a less general situation, was given in \cite{henclm1}).
\begin{proposition}
\label{Propo inv}
Assume $k+l=n-1$ and let $\gamma_1:\Sph^k\to\R^n$ and $\gamma_2:\Sph^l\to\R^n$ be continuous maps such that $\gamma_1(\Sph^k)\cap \gamma_2(\Sph^l)=\varnothing$. Let $h:\R^n\to\R^n$ be a homeomorphism. Then
\begin{itemize}
    \item if $h$ is sense-preserving, then $\ell(h\circ\gamma_1,h\circ\gamma_2)=\ell(\gamma_1,\gamma_2)$;
    \item if $h$ is sense-reversing, then $\ell(h\circ\gamma_1,h\circ\gamma_2)=-\ell(\gamma_1,\gamma_2)$.
\end{itemize}
\end{proposition}
\begin{proof}
We begin by fixing some notation.
Let $\bar{\gamma}_1\in C(\bbbb^{k+1},\R^n)$ be any extension of $\gamma_1$, i.e. $\bar{\gamma}_1|_{\partial\bbbb^{k+1}}=\gamma_1$. Let $A=\bbbb^{k+1}\times\Sph^l$ be a full torus, embedded smoothly in $\R^n$ in a way that the orientation of the boundary of embedded $A$ by the outward normal vector is consistent with the orientation of $\Sph^k\times\Sph^l$,
and define $F:A\to\R^n$, $F(x,y)=\bg_1(x)-\gamma_2(y)$.

In the proof, we shall need a simple lemma, connecting the linking number with the degree of the non-normalized Gauss map $F$.
\begin{lemma}
\label{PL1}
With the above notation,
$$
\ell(\gamma_1,\gamma_2)=\deg(F,A,0).
$$
\end{lemma}
\begin{proof}[Proof of Lemma \ref{PL1}]
The claim would follow from Proposition \ref{propo46}, if $F$ mapped $\partial A$ to $\Sph^{n-1}$, not to $\R^n$ (because $F=\gamma_1-\gamma_2$ on $\partial A$).
Note, however, that $F(\partial A)\subset \R^n\setminus\{0\}$. Set $d=\dist (\gamma_1(\Sph^k),\gamma_2(\Sph^l))=\dist (F(\partial A),0)$, and take $\phi:[0,\infty)\to[0,\infty)$ to be a smooth, positive function such that
$$
\phi(s)=\begin{cases}1&\text{ for }s\leq d/2,\\
s &\text{ for }s\geq d.
\end{cases}
$$
Then
$$
h_t(z)=\frac{z}{\phi(t|z|)}, \quad t\in[0,1],
$$
is a homotopy connecting $\Id:\R^n\to \R^n$ with a mapping which is identity on $\bbbb^n(0,d/2)$ and a projection $z\mapsto z/|z|$ outside $\bbbb^n(0,d)$. Obviously, $0\not\in (h_t\circ F)(\partial A)$ for any $t\in [0,1]$, thus
$$
\deg(F,A,0)=\deg(h_0\circ F, A, 0)=\deg(h_1\circ F, A, 0).
$$
However,
$$
(h_1\circ F)|_{\partial A}(x,y)=\frac{\gamma_1(x)-\gamma_2(y)}{|\gamma_1(x)-\gamma_2(y)|}
$$
is the Gauss map whose degree, by definition, equals $\ell(\gamma_1,\gamma_2)$.
Thus Proposition~\ref{propo46} yields
$$
\ell(\gamma_1,\gamma_2)=\deg(h_1\circ F)|_{\partial A}=
\deg(h_1\circ F,A,0)=\deg (F,A,0).
$$
\end{proof}

Assume now $h$ is sense-preserving (the case of $h$ sense-reversing is treated in the same way).

Let $G_t:A\times[0,1]\to\R^n$,
$$
G_t(x,y)=h(t\bg_1(x))-h(\gamma_2(y)-(1-t)\bg_1(x)).
$$
Then $G_1(x,y)=(h\circ\bg_1)(x)-(h\circ\gamma_2)(y)$, thus by Lemma~\ref{PL1} applied to $h\circ\gamma_1$ and $h\circ\gamma_2$ in place of $\gamma_1$ and $\gamma_2$,
$$
\ell(h\circ\gamma_1,h\circ\gamma_2)=\deg(G_1,A,0).
$$
We have
$$
G_0(x,y)=h(0)-h(\gamma_2(y)-\bg_1(x))=h(0)-h(-F(x,y))=h(0)+(-\Id)\circ h\circ(-\Id)\circ F(x,y).
$$
Note also that $G_t(x,y)=0$ if and only if $\bg_1(x)=\gamma_2(y)$, which is not possible if $(x,y)\in\partial A$, thus $0\not\in G_t(\partial A)$ for any $t\in[0,1]$, and Remark~\ref{R2} yields
\begin{equation*}
\begin{split}
\ell(h\circ\gamma_1,h\circ\gamma_2)&=\deg(G_1,A,0)=\deg(G_0,A,0)\\
&=\deg(h(0)+(-\Id)\circ h\circ(-\Id)\circ F, A,0)\\
&=\deg(F,A,0)=\ell(\gamma_1,\gamma_2),
\end{split}
\end{equation*}
because $z\mapsto h(0)+(-\Id)\circ h\circ(-\Id)(z)$ is a sense-preserving homeomorphism of $\R^n$ which maps $0$ to $0$ and we can apply Proposition \ref{propo31}.

\end{proof}

\section{Proofs of Theorems~\ref{T1} and~\ref{T2}}
\label{Proof1and2}
For $n=1$, the claim of Theorem~\ref{T1} is obvious: a sense-preserving homeomorphism of an open subset of a real line is an increasing function, thus it is differentiable a.e. and its derivative is non-negative.
In dimension $n=2$, every $W^{1,1}$ homeomorphism is again a.e. differentiable (see \cite{Menchoff, GehringL}) and its weak Jacobian coincides with its classical one a.e. The sign of the latter reflects whether the homeomorphism preserves or reverses the local orientation, and thus for a sense preserving homeomorphism $f$ we have $J_f\geq 0$ a.e.

Assume $n\geq 3$. To simplify the notation we shall write $\nu=n-1-[n/2]$.
We will argue by contradiction: assume that the set $\{J_f<0\}$ has positive measure.

Pick a $p$-good point $x_o\in \{J_f<0\}$ for $f$ (in the sense of Definition \ref{D12}) and consider the blow-up $f_r$ of $f$ at $x_o$:
$$
f_r(x)=\frac{f(x_o+rx)-f(x_o)}{r}.
$$
According to Lemma \ref{T4}, $f_r\to f_0$ in $W^{1,p}(\bbbb^n,\bbbr^n)$ as $r\to 0$, where $f_0(x)=Df(x_o)x$.
Note that $f_0$ is a linear, orientation reversing isomorphism.

We pick in $\bbbb^n$ two solid tori, i.e. smooth embeddings $\iota_1:\bbbb^{[n/2]+1}\times\Sph^\nu\to\bbbb^n$ and $\iota_2:\bbbb^{\nu+1}\times\Sph^{[n/2]}\to\bbbb^n$, such that for any $x\in \bbbb^{[n/2]+1}$ and $y\in \bbbb^{\nu+1}$ the embedded spheres $\Sph_x^{\nu}=\iota_1(\{x\}\times \Sph^{\nu})$ and $\Sph_y^{[n/2]}=\iota_2(\{y\}\times \Sph^{[n/2]})$ are linked, with linking number $\ell(\Sph_x^{\nu},\Sph_y^{[n/2]})=+1$ (see the next section for a particular construction).

By Lemma \ref{T10}, we can choose a sequence $r_j\searrow 0$ and particular $x\in \bbbb^{[n/2]+1}$ and $y\in \bbbb^{\nu+1}$ such that $f_j=f_{r_j}$ converge to $f_0$ in $W^{1,p}$ on $\Sph_x^{\nu}\cup\Sph_y^{[n/2]}$. If $n>3$, then $p>[n/2]\geq \nu$, and by the Sobolev-Morrey embedding theorem $f_j$ converge to $f_0$ uniformly on $\Sph_x^{\nu}\cup\Sph_y^{[n/2]}$. If $n=3$ and $p=1$, the $W^{1,1}$ convergence of $f_k$ on the sum of circles $\Sph_x^{\nu} \cup \Sph_y^{[n/2]}$ implies uniform con\-ver\-gence as well.

Since $f_j$ converge to $f_0$ uniformly on $\Sph_x^{\nu}\cup\Sph_y^{[n/2]}$, $f_j$ are homotopic to $f_0$ for $j$ sufficiently large, and the image of each sphere in that homotopy does not intersect the image of the other (the image of $\Sph_x^{\nu}$ in that homotopy stays in a small tubular neighborhood of $f_0(\Sph_x^{\nu})$, and, similarly, the image of $\Sph_y^{[n/2]}$ in that homotopy stays near $f_0(\Sph_y^{[n/2]})$). Thus
$$
\ell(f_j(\Sph_x^{\nu}),f_j(\Sph_y^{[n/2]}))=\ell(f_0(\Sph_x^{\nu}),f_0(\Sph_y^{[n/2]}))=-\ell(\Sph_x^{\nu},\Sph_y^{[n/2]})=-1,
$$
since $f_0$ is a linear, orientation reversing homeomorphism.

However, each $f_j$, as a translation and rescaling of an orientation preserving homeomorphism $f$, is again an orientation preserving homeomorphism, thus Proposition~\ref{Propo inv} yields
$$
\ell(f_j(\Sph_x^{\nu}),f_j(\Sph_y^{[n/2]}))=\ell(\Sph_x^{\nu},\Sph_y^{[n/2]})=+1,
$$
which gives the desired contradiction.
\begin{remark}
Note that the proofs of Theorems~\ref{T1} and~\ref{T2} required only a few results from Section~\ref{Prelim}, namely Lemma~\ref{T4}, Lemma~\ref{T10} and Proposition~\ref{propo46} (i.e. the whole Section~\ref{Sec2.5}). However, the proof of Theorem~\ref{main} will require the whole content of Section~\ref{Prelim}, that is, in addition to results needed for the proofs of Theorems~\ref{T1} and~\ref{T2}, we will need Lemma~\ref{T9} and Proposition~\ref{L27}.
\end{remark}

\section{Proof of Theorem \ref{main}.}
\label{Proofof5}
Throughout the proof, we shall assume that the assumption a) holds, i.e. $f$ maps almost all $[n/2]$-dimensional spheres into sets of $[n/2]+1$ dimensional Hausdorff measure zero. The case when b) holds is treated in the same way. We simply need to exchange $f$ and $f^{-1}$ in the proof below.

We argue by contradiction: assume that the set $\{J_f<0\}$ has positive measure.

To simplify the notation we shall write $\nu=n-1-[n/2]$.
According to a local version of Lemma~\ref{T9} for homeomorphisms on domains instead of $\R^n$, we can find $x_o$ and
a linear transformation $A\in GL(n)$ with $\det A>0$ such that the sense preserving homeomorphism
$g=A\circ f$ satisfies
$$
\lim_{r\to 0^+}\Vert g_r-\mathcal{R}\Vert_{W^{1,[n/2]}(\bbbb^n,\R^n)}=
\lim_{r\to 0^+}\Vert (g_r)^{-1}-\mathcal{R}\Vert_{W^{1,\nu}(\bbbb^n,\R^n)}=0.
$$
Let $\bbbb_1=\overbar{\bbbb}^{[n/2]+1}$, $\bbbb_2=\overbar{\bbbb}^{\nu+1}$ be closed unit balls.
Note that the manifolds $\bbbb_1\times\Sph^\nu$ and $\bbbb_2\times\Sph^{[n/2]}$ have dimension $n$.
Let
$\iota_1:\bbbb_1\times\Sph^\nu\hookrightarrow\bbbb^n$ and  $\iota_2:\bbbb_2\times\Sph^{[n/2]}\hookrightarrow\bbbb^n$
be smooth embeddings, smooth up to the boundary. According to Lemma~\ref{T10} applied to the family $g_r\circ\iota_2 $ and then
for the second time to the family $(g_r)^{-1}\circ\iota_1$, we can find a sequence $r_k\searrow 0$ such that
the mappings $g_k:=g_{r_k}$ satisfy:

There are compact sets $K_1\subset \bbbb_1$ and $K_2\subset \bbbb_2$ of positive measure such that
\begin{equation}
\label{eq16}
\lim_{k\to\infty}
\sup_{x\in K_1}\Vert (g_k)^{-1}\circ\iota_1-\mathcal{R}\circ\iota_1\Vert_{W^{1,\nu}(\{x\}\times\Sph^\nu)}=0
\end{equation}
and
\begin{equation}
\label{eq24}
\lim_{k\to\infty}
\sup_{y\in K_2}
\Vert g_k\circ\iota_2-\mathcal{R}\circ\iota_2\Vert_{W^{1,[n/2]}(\{y\}\times\Sph^{[n/2]})}
=0.
\end{equation}
We shall define the embeddings $\iota_1$ and $\iota_2$ explicitly, to make sure that the embedded spheres
$\iota_1(\{x\}\times\Sph^\nu)$ and $\iota_2(\{y\}\times\Sph^{[n/2]})$ are linked with the linking number $1$.

\begin{itemize}
\item for $x\in \bbbb_1\subset\R^{[n/2]+1}$, $\sigma\in \Sph^\nu\subset\R^{\nu+1}$, we set
$$
\iota_1(x,\sigma)=(\frac{5+x_1}{10}\sigma_1,\ldots,\frac{5+x_1}{10}\sigma_\nu,\frac{5+x_1}{10}\sigma_{\nu+1}-\frac{1}{4},\frac{x_2}{10}\ldots,\frac{x_{[n/2]+1}}{10}),
$$
thus the image of $\iota_1$ is the full torus with its core sphere lying in the hyperplane of the first $\nu+1$ coordinates,
\item for $y\in \bbbb_2\subset\R^{\nu+1}$ and $\rho\in\Sph^{[n/2]}\subset \R^{[n/2]+1}$, we set $$
\iota_2(y,\rho)=
(\frac{y_1}{10},\ldots,\frac{y_\nu}{10},\frac{5+y_{\nu+1}}{10}\rho_1+\frac{1}{4},\frac{5+y_{\nu+1}}{10}\rho_2,\ldots,\frac{5+y_{\nu+1}}{10}\rho_{[n/2]+1}),$$ thus the image of $\iota_2$ is the full torus with its core sphere lying in the hyperplane of the last $[n/2]+1$ coordinates.
\end{itemize}
To simplify the notation, as in the previous section we shall write $\Sph^\nu_x=\iota_1(\{x\}\times\Sph^\nu)$, $\Sph^{[n/2]}_y=\iota_2(\{y\}\times\Sph^{[n/2]}))$.

Since the assumption a) holds, i.e., $f$ maps almost every sphere of dimension $[n/2]$ to a set of $([n/2]+1)$-Hausdorff measure zero,
it easily follows that $g_k$, for every $k$, has the same property. We may thus assume, possibly shrinking $K_2$, that for every $y\in K_2$ and any $k$ the set $g_k(\Sph_y^{[n/2]})$ has  $([n/2]+1)$-Hausdorff measure zero.
(Similarly, if b) holds, i.e., $f^{-1}$ maps almost every sphere of dimension $\nu$ to a set of $(\nu+1)$-Hausdorff measure zero, we can assume that for every $x\in K_1$ and any $k$ the set $g^{-1}_k(\Sph_x^{\nu})$ has  $(\nu+1)$-Hausdorff measure zero.)

The images of $\iota_1$ and $\iota_2$ are two full, disjoint, linked tori, and for each $x\in \bbbb_1$ and $y\in \bbbb_2$, $\Sph^\nu_x$ and $\Sph^{[n/2]}_y$ are two linked spheres. We choose orientations of the spheres $\Sph^\nu$ and $\Sph^{[n/2]}$ so that the linking number equals
$$
\ell(\Sph^\nu_x,\Sph^{[n/2]}_y):=\ell\big(\iota_1|_{\{x\}\times\Sph^\nu},\iota_2|_{\{y\}\times\Sph^{[n/2]}}\big)=+1.
$$
The first equality mans that for all $x\in\bbbb_1$ and all $y\in\bbbb_2$
we equip $\Sph_x^\nu$ and $\Sph_y^{[n/2]}$ with orientations so that the diffeomorphisms
$$
\iota_1|_{\{x\}\times\Sph^\nu}:\{x\}\times\Sph^\nu \to \Sph_x^\nu
\quad
\text{and}
\quad
\iota_2|_{\{y\}\times\Sph^{[n/2]}}:\{y\}\times\Sph^{[n/2]}\to \Sph_y^{[n/2]}
$$
are orientation preserving.

The linear transformation $\cR$ is the reflection in the last coordinate. Since the spheres $\Sph_y^{[n/2]}$ are centered at a point lying in $\R^{\nu+1}\times\{0\}$, the reflection $\cR$ preserves the center and hence $\cR(\Sph^{[n/2]}_y)=\Sph^{[n/2]}_y$. However, the reflection $\cR$ changes the orientation of the sphere $\Sph^{[n/2]}_y$.
More precisely, the mapping
\begin{equation}
\label{eq21}
\cR:\Sph_y^{[n/2]}\to\Sph_y^{[n/2]}
\quad
\text{has degree $-1$.}
\end{equation}
By $\cR(\Sph_y^{[n/2]})$ we denote the sphere $\Sph^{[n/2]}_y$ with the opposite orientation, so
the mapping $\cR:\Sph_y^{[n/2]}\to\cR(\Sph_y^{[n/2]})$ is orientation preserving.
In particular, \eqref{eq21} implies that
\begin{equation}
\label{eq22}
\ell(\Sph_x^\nu,\cR(\Sph_y^{[n/2]}))=-1.
\end{equation}

On the other hand, the spheres $\Sph^\nu_x$ are centered at points lying in $\{0\}\times\R^{[n/2]+1}$ and the last coordinate of the center is
$x_{[n/2]+1}/10$, where $x=(x_1,\ldots,x_{[n/2]+1})$, so
$\cR(\Sph^\nu_x)=\Sph^\nu_{\tilde{x}}$ where $\tilde{x}=(x_1,\ldots,x_{[n/2]},-x_{[n/2]+1})$.
Also, the orientation of $\cR(\Sph^\nu_x)$ is the same as that of
$\Sph^\nu_{\tilde{x}}$. More precisely, the mapping
$\cR:\Sph_x^\nu\to \Sph_{\tilde{x}}^\nu$ is orientation preserving, so
$\cR(\Sph_x^\nu)$ denotes the sphere $\Sph_{\tilde{x}}^\nu$ with the original orientation. In particular,
\begin{equation}
\label{eq23}
\ell(\cR(\Sph_x^\nu),S_y^{[n/2]})=\ell(\Sph_{\tilde{x}}^\nu,\Sph_y^{[n/2]})=+1.
\end{equation}

By Proposition~\ref{L27} and Remark~\ref{rem1}, for each $k\in\bbbn$ we may define a continuous map
$$
H_{1,k}:K_1\times\Sph^\nu\times [0,1]\to\R^n
$$
such that for each $x\in K_1$, $H_{1,k}(x,\cdot,\cdot)$ is a homotopy between
$$
H_{1,k}(x,\cdot,1)=g^{-1}_k\circ\iota_1|_{\{ x\}\times\Sph^\nu}
\quad
\text{and}
\quad
H_{1,k}(x,\cdot,0)=\cR\circ\iota_1|_{\{ x\}\times\Sph^\nu}.
$$
Moreover, Proposition~\ref{L27} along with \eqref{eq16} yields
$$
\sup_{x\in K_1}\cH^{\nu+1}(H_{1,k}(\{x\}\times\Sph^\nu\times [0,1))\to 0
\quad
\text{as $k\to\infty$.}
$$
Therefore by taking a suitable subsequence of $g_k^{-1}$ (still denoted by $g_k^{-1}$) we may require that
\begin{equation}
\label{eq17}
\sup_{x\in K_1}\mathcal{H}^{\nu+1}(H_{1,k}(\{x\}\times \Sph^\nu \times [0,1)))< \frac{1}{2^k},
\quad
\text{for all $k$.}
\end{equation}
Likewise, we define a continuous map
$$
H_{2,k}:K_2\times\Sph^{[n/2]}\times [0,1]\to\R^n
$$
such that for all $y\in K_2$,
$H_{2,k}(y,\cdot,\cdot)$ is a homotopy between
$$
H_{2,k}(y,\cdot,1)=g_k\circ\iota_2|_{\{ y\}\times\Sph^{[n/2]}}
\quad
\text{and}
\quad
H_{2,k}(y,\cdot,0)=\cR\circ\iota_2|_{\{ y\}\times\Sph^{[n/2]}}.
$$
Again, Proposition~\ref{L27} along with \eqref{eq24} yields
$$
\sup_{y\in K_2}\cH^{[n/2]+1}(H_{2,k}(\{y\}\times\Sph^{[n/2]}\times [0,1))\to 0
\quad
\text{as $k\to\infty$.}
$$
Since for every $y\in K_2$ we have
$$
\cH^{[n/2]+1}(H_{2,k}(\{y\}\times\Sph^{[n/2]}\times\{ 1\})=
\cH^{[n/2]+1}(g_k(\Sph_y^{[n/2]}))=0,
$$
by taking a suitable subsequence we may require that
\begin{equation}
\label{eq18}
\sup_{y\in K_2}\cH^{[n/2]+1}(H_{2,k}(\{ y\}\times\Sph^{[n/2]}\times [0,1]))<\frac{1}{2^k}
\quad
\text{for all $k$.}
\end{equation}
Note that this is a stronger condition than \eqref{eq17} in the sense that now we have the estimate for the whole interval $[0,1]$, while in \eqref{eq17} we only have the estimate for the interval $[0,1)$.

We prove the following:
\begin{lemma}
\label{lem:no inters}
For every $y\in K_2$, for almost every $x\in K_1$
\begin{equation}
\label{nointer1}
\exists_{m\in \N}\ \forall_{k\geq m}
\quad
\Sph_x^\nu\cap H_{2,k}(\{y\}\times\Sph^{[n/2]}\times[0,1])=\varnothing.
\end{equation}
Also, for every $x\in K_1$, for almost every $y\in K_2$
\begin{equation}
\label{nointer2}
\exists_{m\in \N}\ \forall_{k\geq m}
\quad
\Sph_y^{[n/2]}\cap  H_{1,k}(\{x\}\times\Sph^{\nu}\times[0,1))=\varnothing.
\end{equation}
\end{lemma}

\begin{proof}[Proof of Lemma \ref{lem:no inters}]
We will only prove \eqref{nointer1} since the proof of \eqref{nointer2} follows from the same reasoning.
Fix $y\in K_2$. We want to show that for almost all $x\in K_1$, \eqref{nointer1} is satisfied.

We define a projection of the embedded full torus $\iota_1(\bbbb_1\times\Sph^\nu)$ onto $\bbbb_1$ by
$$
\pi:\iota_1(\bbbb_1\times\Sph^\nu)\to \bbbb_1,
\qquad
\pi(\iota(x,\sigma))=x.
$$
It is easy to see that $\pi$ is  $10$-Lipschitz and hence it increases the Hausdorff measure $\cH^{[n/2]+1}$ of a set at most by a constant factor $C(n)=10^{[n/2]+1}$. Let
$$
T_k(y)=H_{2,k}(\{ y\}\times\Sph^{[n/2]}\times [0,1])\cap\iota_1(\bbbb_1\times\Sph^\nu)
$$
be a part of the image of the homotopy $H_{2,k}(y,\cdot,\cdot)$ that is contained in the domain of the projection $\pi$.
Estimate \eqref{eq18} and the fact that $\pi$ increases the Hausdorff measure by at most $C(n)$ imply
$$
\cH^{[n/2]+1}\left(\bigcup_{k=m}^\infty \pi(T_k(y))\right)<
\sum_{k=m}^\infty C(n) 2^{-k}=C(n)2^{-m+1}\to 0
\quad
\text{as $m\to\infty$,}
$$
so
$$
\cH^{[n/2]+1}\left(\bigcap_{m=1}^\infty\bigcup_{k=m}^\infty \pi(T_k(y))\right)=0.
$$
To complete the proof of \eqref{nointer1} it suffices to show that \eqref{nointer1} is satisfied by all
\begin{equation}
\label{eq19}
x\in K_1\setminus \bigcap_{m=1}^\infty\bigcup_{k=m}^\infty \pi(T_k(y))=
\bigcup_{m=1}^\infty\bigcap_{k=m}^\infty (K_1\setminus \pi(T_k(y))).
\end{equation}
Note that if $x\in K_1\setminus \pi(T_k(y))$, then
$$
\pi(\Sph_x^\nu\cap T_k(y))\subset \pi(\Sph_x^\nu)\cap \pi(T_k(y))=
\{ x\}\cap \pi(T_k(y))=\varnothing,
$$
so
\begin{equation}
\label{eq20}
\Sph_x^\nu\cap H_{2,k}(\{ y\}\times\Sph^{[n/2]}\times [0,1])=
\Sph_x^\nu\cap T_k(y)=\varnothing.
\end{equation}
Therefore if $x$ belongs to the set \eqref{eq19}, then
$$
\exists_{m\in \N}\ \forall_{k\geq m} \quad
x\in K_1\setminus \pi(T_k(y)).
$$
Since the condition $x\in K_1\setminus \pi(T_k(y))$ implies \eqref{eq20}, claim
\eqref{nointer1} follows.

\end{proof}

To finish the proof of Theorem \ref{main}, we want to choose $x\in K_1$ and $y\in K_2$ in such a way that
\begin{itemize}
\item[i)] there exists $m$ such that
$$
\Sph_x^\nu\cap \bigcup_{k=m}^\infty
H_{2,k}(\{y\}\times\Sph^{[n/2]}\times [0,1])=\varnothing,
$$
i.e., the sphere $\Sph_x^\nu$ avoids, for all sufficiently large $k$, the image of the homotopy ${H_{2,k}(y,\cdot,\cdot)\circ\iota_2^{-1}}$
joining $g_k|_{\Sph_y^{[n/2]}}$ with $\cR|_{\Sph_y^{[n/2]}}$,\\
and simultaneously
\item[ii)] there exists $m$ such that
$$
\Sph_y^{[n/2]}\cap \bigcup_{k=m}^\infty
H_{1,k}(\{ x\}\times\Sph^\nu\times [0,1))=\varnothing,
$$
i.e., the sphere $\Sph_y^{[n/2]}$ avoid, for all sufficiently large $k$, the image of the homotopy
$H_{1,k}(x,\cdot,\cdot)\circ\iota_1^{-1}$
joining $g^{-1}_k|_{\Sph_x^\nu}$ with $\cR|_{\Sph_x^\nu}$,
except possibly at the endpoint: we do not rule out yet that $\Sph_y^{[n/2]}\cap g^{-1}_k(\Sph_x^\nu)\neq \varnothing$.
\end{itemize}
Assume we have chosen $x$ and $y$ satisfying conditions i) and ii) above. Then, for all sufficiently large $k$, $\Sph_x^\nu\cap g_k(\Sph_y^{[n/2]})=\varnothing$, and since $g_k^{-1}$ is a homeomorphism, this immediately implies that $\Sph_y^{[n/2]}\cap g^{-1}_k(\Sph_x^\nu)= \varnothing$. Therefore, for all sufficiently large $k$, the sphere $\Sph_y^{[n/2]}$ avoids the image of the whole homotopy joining $g^{-1}_k|_{\Sph_x^\nu}$ with $\cR|_{\Sph_x^\nu}$, including the endpoint:
\begin{itemize}
\item[ii')] there exists $m$ such that
$$
\Sph_y^{[n/2]}\cap \bigcup_{k=m}^\infty
H_{1,k}(\{ x\}\times\Sph^\nu\times [0,1])=\varnothing.
$$
\end{itemize}

Denote by $A_1\subset K_1\times K_2$ the set of all $(x,y)$ satisfying the condition i) above. By Lemma~\ref{lem:no inters}, \eqref{nointer1}, for every $y_o\in K_2$ the section $A_1\cap \{(x,y_o)~~:~~a\in K_1\}$ is of full measure, and thus, by Fubini's theorem, $A_1$ is of full measure in $K_1\times K_2$, provided that $A_1$ is a measurable set. Similarly, the set $A_2\subset K_1\times K_2$ of these $(x,y)$, which satisfy the condition ii), if measurable, is of full measure in $K_1\times K_2$. We shall leave the issue of measurability of $A_1$ and $A_2$ and address it at the end of the proof.
Since $A_1$ and $A_2$ are of full measure, their intersection is not empty and we can find $x$ and $y$ simultaneously satisfying the conditions i) and ii) and hence conditions i) and ii').

In particular, there is $x\in K_1$, $y\in K_2$ and $k\in\N$ such that
$$
\Sph_x^\nu\cap
H_{2,k}(\{y\}\times\Sph^{[n/2]}\times [0,1])=
\Sph_y^{[n/2]}\cap
H_{1,k}(\{ x\}\times\Sph^\nu\times [0,1])=
\varnothing.
$$
We fix such a point $(x,y)\in K_1\times K_2$ and we look at linking numbers of
spheres and their images in $g_k$, $g_k^{-1}$ and $\cR$.

The mappings $g_k^{-1}|_{\Sph_x^\nu}$ and $\cR|_{\Sph_x^\nu}$ are homotopic
(with $H_{1,k}(x,\cdot,\cdot)\circ\iota_1^{-1}$ providing the homotopy), and the image of the homotopy does not intersect $\Sph_y^{[n/2]}$. This and \eqref{eq23} yield
\begin{equation*}
+1=\ell(\Sph_{\tilde{x}}^\nu,\Sph_y^{[n/2]})=\ell(\cR(\Sph_x^\nu),\Sph_y^{[n/2]})=
\ell(g_k^{-1}(\Sph_x^\nu),\Sph_y^{[n/2]})
\end{equation*}
and, since $g_k$ is a sense preserving homeomorphism,
$$
+1=\ell(g_k^{-1}(\Sph_x^\nu),\Sph_y^{[n/2]})=\ell(\Sph_x^\nu,g_k(\Sph_y^{[n/2]})).
$$
Next, using the homotopy between $\cR|_{\Sph_y^{[n/2]}}$ and $g_k|_{\Sph_y^{[n/2]}}$, given by $H_{2,k}(y,\cdot,\cdot)\circ\iota_2^{-1}$, we have
$$
+1=\ell(\Sph_x^\nu,g_k(\Sph_y^{[n/2]}))=\ell(\Sph_x^\nu,\cR(\Sph_y^{[n/2]}))=-1,
$$
where the last equality follows from \eqref{eq22}.
This gives the desired contradiction and finishes the proof, except for the set aside problem of measurability of the sets $A_1$ and $A_2$.

Since the proof of measurability of both sets follows exactly the same scheme, we shall prove only that $A_1$ is measurable.

If we write
$$
W_k=\{(x,y)\in K_1\times K_2~~:~~\Sph_x^\nu \cap H_{2,k}(\{ y\}\times\Sph^{[n/2]}\times [0,1])=\varnothing\},
$$
then $A_1=\bigcup_{m=1}^\infty \bigcap_{k=m}^\infty W_k$, thus to prove measurability of $A_1$, it suffices to prove it for $W_k$.

Let
$$F_k=(\iota_1, H_{2,k})\colon K_1\times \Sph^\nu\times K_2\times \Sph^{[n/2]}\times [0,1]\to \R^n\times \R^n$$ and denote by
$\Delta=\{(x,x)~~:~~x\in\R^n\}$ the diagonal in $\R^n\times \R^n$.
Then, since $F_k$ is continuous, $K_1\times \Sph^\nu\times K_2\times \Sph^{[n/2]}\times [0,1]$ is compact and $\Delta$ is closed, the set $F_k^{-1}(\Delta)\subset \R^{2n+1}$ is compact. Let now
$$
\Pi:K_1\times \Sph^\nu\times K_2\times \Sph^{[n/2]}\times [0,1]\to K_1\times K_2
$$
be the projection on the first and third factors.
The set $\Pi(F_k^{-1}(\Delta))$ is a compact subset of $K_1\times K_2$.
We have
\begin{equation}
\begin{split}
(x,y)&\in \Pi(F_k^{-1}(\Delta))\\
&\Leftrightarrow \text{ there exist }\sigma\in \Sph^\nu,\,\rho\in \Sph^{[n/2]},\, t\in [0,1] \text{ with } (x,\sigma,y,\rho,t)\in F_k^{-1}(\Delta)\\
& \Leftrightarrow \exists_{\sigma,\rho,t} \quad F_k(x,\sigma,y,\rho,t)\in\Delta\\
& \Leftrightarrow \exists_{\sigma,\rho,t}\quad \iota_1(x,\sigma)=H_{2,k}(y,\rho,t)\\
&\Leftrightarrow \Sph_x^\nu\cap  H_{2,k}(\{y\}\times\Sph^{[n/2]}\times [0,1])\neq \varnothing,
\end{split}
\end{equation}
which shows that $W_k=(K_1\times K_2)\setminus \Pi(F_k^{-1}(\Delta))$, and that $W_k$ is an open (and thus measurable) subset of $K_1\times K_2$. This concludes the proof of measurability of $A_1$ and the proof of Theorem \ref{main}.
\begin{remark}
Under the assumptions of Corollaries \ref{T14} and \ref{T15}, the proof simplifies greatly. Recall that in these corollaries we assume $n=2m$, thus $\nu=n-[n/2]-1=m-1$. If we assume $f^{-1}\in W^{1,m-1+\varepsilon}$, then $g_k^{-1}\to \cR$ in $W^{1,m-1+\varepsilon}$, and by the Morrey-Sobolev imbedding, on almost every sphere $\Sph_x^{\nu}$ this convergence is uniform (the same conclusion holds if we assume $Df^{-1}\in L^{m-1,1}_{loc}$). We can set the homotopy between $\cR|_{\Sph_x^{\nu}}$ and  $g_k^{-1}|_{\Sph_x^{\nu}}$ to be $H_1,k(x,\sigma,t)=tg_k^{-1}(\iota_1(x,\sigma)+{(1-t)\cR(\iota_1(x,\sigma))}$; then for a.e. $x\in \bbbb_1$ and sufficiently large $k$ the whole image of the homotopy $H_1(x,\cdot,\cdot)$ lies in the torus $\iota_1(\bbbb_1\times \Sph^{[n/2]})$, thus for any $y\in\bbbb_2$ this image does not intersect $\Sph_y^{[n/2]}$:
\begin{equation}
\label{rem:noint1}
\exists_l\ \forall_{k>l}\quad \Sph_y^{[n/2]}\cap H_1(\{x\}\times \Sph^\nu\times [0,1])=\varnothing.
\end{equation}
Fix any $y\in \bbbb_2$. Denoting by $\tilde{Y}_k(y)$ the projection of $\tilde{T}_k(y)$ onto the cross section $\iota_1(\bbbb_1\times \{\sigma_o\})$ (in analogy to $Y_k(y)$), we see that the $(m+1)$-dimensional Hausdorff measure of $\tilde{Y}_k(y)$ tends to zero. We can thus find $x\in \bbbb_1$  such that \eqref{rem:noint1} holds and $\iota_1(x,\sigma_o)\cap \tilde{Y}_k(y)=\varnothing$ for some $k>l$. Then $\Sph_x^\nu$ does not intersect the image of the homotopy joining $\cR|_{\Sph_y^{[n/2]}}$ with $g_k|_{\Sph_y^{[n/2]}}$, except possibly at the endpoint -- we have not ruled out that $\Sph_x^\nu\cap g_k(\Sph_y^{[n/2]})\neq\varnothing$. We have thus found $x$ and $y$ satisfying conditions i) and ii) in the proof of Theorem \ref{main} (although with $x$ and $y$ exchanged) and we may conclude the proof as it is done there.
\end{remark}

\end{document}